\documentclass[a4paper,10pt]{amsart}
\usepackage{microtype}
\usepackage[utf8]{inputenc}
\usepackage[T1]{fontenc}
\usepackage{amsfonts, amsmath, amssymb, amsthm, mathtools}
\usepackage{xcolor}
\definecolor{darkblue}{rgb}{0,0,0.6}
\usepackage{tikz}
\usetikzlibrary{matrix,arrows,trees}
\usepackage{tikz-cd}
\usepackage[ocgcolorlinks,colorlinks=true, citecolor=darkblue, filecolor=darkblue, linkcolor=darkblue, urlcolor=darkblue]{hyperref}
\usepackage[nameinlink, capitalise]{cleveref}

\usepackage{eucal}                  

\DeclareRobustCommand{\SkipTocEntry}[5]{}

\hypersetup{bookmarksopen=true}

\newcommand*\cocolon{%
	\nobreak
	\mskip6mu plus1mu
	\mathpunct{}%
	\nonscript
	\mkern-\thinmuskip
	{:}%
	\mskip2mu
	\relax
}

\numberwithin{equation}{section}

\swapnumbers
\theoremstyle{definition}
\newtheorem{defn}[equation]{Definition}
\newtheorem{ex}[equation]{Example}
\newtheorem{const}[equation]{Construction}
\newtheorem{rem}[equation]{Remark}
\newtheorem{theorem}[equation]{Theorem}
\newtheorem{prop}[equation]{Proposition}
\newtheorem{lem}[equation]{Lemma}
\newtheorem{cor}[equation]{Corollary}

\renewcommand{\epsilon}{\varepsilon}
\renewcommand{\theta}{\vartheta}
\renewcommand{\phi}{\varphi}

\newcommand{\bbK}{\mathbb{K}}

\newcommand{\bbN}{\mathbb{N}}

\newcommand{\bbS}{\mathbb{S}}

\newcommand{\bbZ}{\mathbb{Z}}

\newcommand{\cA}{\mathcal{A}}

\newcommand{\cC}{\mathcal{C}}
\newcommand{\cD}{\mathcal{D}}
\newcommand{\cE}{\mathcal{E}}
\newcommand{\cF}{\mathcal{F}}

\newcommand{\cO}{\mathcal{O}}
\newcommand{\cP}{\mathcal{P}}
\newcommand{\cQ}{\mathcal{Q}}

\newcommand{\cU}{\mathcal{U}}

\newcommand{\cX}{\mathcal{X}}
\newcommand{\cY}{\mathcal{Y}}

\DeclareMathOperator*{\colim}{colim}
\DeclareMathOperator{\cofib}{cofib}
\DeclareMathOperator{\fib}{fib}
\DeclareMathOperator{\Fun}{Fun}

\DeclareMathOperator{\id}{id}
\DeclareMathOperator{\inc}{inc}
\DeclareMathOperator{\Ind}{Ind}

\newcommand{\Ldec}{\mathrm{L}}

\newcommand{\st}{\mathrm{st}}

\newcommand{\Ab}{\mathrm{Ab}}
\newcommand{\catinf}{\mathrm{Cat}}
\newcommand{\catrex}{\catinf^{\mathrm{rex}}}
\newcommand{\catst}{\catinf^{\st}}
\newcommand{\catperf}{\catinf^{\mathrm{perf}}}

\newcommand{\Funex}[1]{\Fun^{\mathrm{ex}}_{#1}}
\newcommand{\op}{\mathrm{op}}
\newcommand{\PrL}{\mathrm{Pr}^\mathrm{L}}
\newcommand{\psh}{\mathcal{P}}

\newcommand{\Sp}{\mathrm{Sp}}
\newcommand{\Spc}{\mathrm{An}}

\newcounter{commentcounter}

\newcommand{\bdd}{\mathrm{b}}

\newcommand{\egr}{\mathrm{eg}}
\newcommand{\exct}{\mathrm{ex}}
\newcommand{\biex}{\mathrm{biex}}

\newcommand{\ing}{\mathrm{in}}
\newcommand{\lex}{\mathrm{lex}}

\newcommand{\lloc}{\mathrm{lloc}}
\newcommand{\rloc}{\mathrm{rloc}}

\newcommand{\sh}{\mathrm{Sh}}

\newcommand{\catprst}{\catinf^{\mathrm{prest}}}

\newcommand{\catheart}{\catinf^\heartsuit}
\newcommand{\derb}{\cD^{\bdd}}
\newcommand{\exact}{\mathrm{Exact}}
\newcommand{\exactperf}{\exact^\mathrm{perf}}
\newcommand{\Groth}{\mathrm{Groth}}
\newcommand{\idem}[1]{{#1}^\natural}

\newcommand{\pair}{\mathrm{Pair}}

\newcommand{\shv}{\mathrm{Shv}}

\DeclareMathOperator{\Acycl}{A}
\DeclareMathOperator{\Alg}{Alg}

\DeclareMathOperator{\coker}{coker}

\DeclareMathOperator{\K}{K}
\DeclareMathOperator{\LEx}{\Lambda}

\DeclareMathOperator{\Nat}{Nat}
\DeclareMathOperator{\Proj}{Proj}
\DeclareMathOperator{\stab}{SW}

\let\sp\Sp

\DeclareMathOperator{\Syn}{Syn}
\newcommand{\fun}{\operatorname{Fun}}

\DeclareFontFamily{U}{min}{}
\DeclareFontShape{U}{min}{m}{n}{<-> dmjhira}{}
\newcommand{\yo}{\text{\usefont{U}{min}{m}{n}\symbol{'110}}}

\title{The presentable stable envelope of an exact category}
\author{Marius Nielsen}
\address{Department of Mathematical Sciences, Norwegian University of Science
and Technology, Trondheim, Norway}
\email{marius.v.b.nielsen@ntnu.no}

\author{Christoph Winges}
\address{Fakult\"at f\"ur Mathematik, Universit\"at Regensburg, 93040 Regensburg, Germany}
\email{christoph.winges@ur.de}

\date{}

\begin{document}

\begin{abstract}
    We prove an analogue of the Gabriel--Quillen embedding theorem for exact $\infty$-categories, giving rise to a presentable version of Klemenc's stable envelope of an exact $\infty$-category.
    Moreover, we construct a symmetric monoidal structure on the $\infty$-category of small exact $\infty$-categories and discuss the multiplicative properties of the Gabriel--Quillen embedding. For $E$ an Adams-type homotopy associative ring spectrum, this allows us to identify the symmetric monoidal $\infty$-category of $E$-based synthetic spectra with the presentable stable envelope of the exact $\infty$-category of compact spectra with finite projective $E$-homology.
    
    In addition, we show that algebraic K-theory, considered as a functor on exact $\infty$-categories, admits a unique delooping as a localising invariant.
\end{abstract}

\maketitle
\tableofcontents

\section{Introduction}
Exact $\infty$-categories were introduced by Barwick in \cite{barwick:heart} generalising the notion of exact categories as introduced by Quillen in \cite{quillen:k-theory}. Every stable $\infty$-category admits the structure of an exact $\infty$-category and Klemenc proves in \cite{klemenc:stablehull} that the functor
\[
    \catst \to \exact
\]
sending a small stable $\infty$-category to its associated exact $\infty$-category admits a left adjoint. Let us denote this functor by $\derb(-)$ and name it the \emph{stable envelope}. Klemenc also shows that the unit map
\[
    \cC \to \derb(\cC)
\]
is fully faithful and exhibits $\cC$ as an extension-closed subcategory of its stable envelope.'
However, $\mathcal{D}^{b}(\mathcal{C})$ is manifestly a small category.

In the case of a small exact $1$-category $\mathcal{E}$,
the \emph{Gabriel--Quillen embedding}
\[
    \mathcal{E}\to \psh_\lex^{\mathrm{Ab}}(\mathcal{E})
\]
exhibits $\cE$ as an extension-closed subcategory of a Grothendieck abelian category. Here $\psh_\lex^{\mathrm{Ab}}(\mathcal{E})$ denotes the category of abelian group-valued left exact presheaves.
See \cite[Section~2]{quillen:k-theory}, \cite[Appendix~A]{TT}, \cite[Appendix~A]{buehler} or \cite[Section~2]{kl:cohomology-exact-categories} for various treatments of this embedding.

In this paper we establish an analogous version of the Gabriel--Quillen embedding pertaining to exact $\infty$-categories. More precisely, we prove:

\begin{theorem}[Gabriel--Quillen embedding theorem, \cref{thm:plex,thm:gabriel-quillen}]
 Let $\cC$ be a small exact $\infty$-category.
 Then the Yoneda embedding $\cC \to \psh_\lex(\cC)$ exhibits $\cC$ as an extension-closed subcategory of a Grothendieck prestable $\infty$-category. 

Furthermore, for every pointed cocomplete $\infty$-category $\cD$, precomposing with the Yoneda embedding induces an equivalence
 \[
        \fun^{\mathrm{L}}(\psh_{\lex}(\cC),\cD) \xrightarrow{\yo^{*}} \fun^{\mathrm{ex}}(\cC,\cD).
 \]
Here $\psh_\lex(\cC)\subseteq \psh(\cC)$ denotes the full subcategory spanned by left exact presheaves.
\end{theorem}

This prompts the following definition.

\begin{defn}
    If $\mathcal{C}$ is an exact $\infty$-category, then the \emph{presentable stable envelope} of $\mathcal{C}$ is defined as
    \[
        \psh_\lex^\st(\cC) := \Sp(\psh_\lex(\cC)).
    \]
\end{defn}

The presentable stable envelope is the initial presentable stable $\infty$-category with an exact functor from $\mathcal{C}$.
This generalises Klemenc's construction as we show that $\Ind(\derb(\cC)) \simeq \psh_\lex^\st(\cC)$, see \cref{cor:ind-derb-is-plex}.

We then proceed to discuss the multiplicative properties of the Gabriel--Quillen embedding and the stable envelope functor in \cref{sec:monoidal}, where we prove the following.

\begin{theorem}[\cref{cor:exact-symm-monoidal}]
    The $\infty$-category $\exact$ of small exact $\infty$-categories and exact functors admits a symmetric monoidal structure with unit $\mathrm{Proj}(\mathbb{S})$, the additive $\infty$-category of finitely generated projective $\bbS$-modules. Moreover, the functors 
    \[
        \psh_\lex \colon \exact \to \Groth \quad\text{and}\quad \psh_\lex^\st \colon \exact \to \PrL_\st
    \]
    refine to symmetric monoidal functors.
\end{theorem}

In \cref{sec:synthetic}, we focus on the $\infty$-category $\Syn_{E}$ of $E$-based synthetic spectra for an Adams-type homotopy associative ring spectrum $E$ introduced by Pstr\k{a}gowski in \cite{pstragowski:synthetic}. The $\infty$-category of synthetic spectra acts as a $1$-parameter deformation of the $\infty$-category of spectra, with respect to the $E$-based Adams spectral sequence. This $\infty$-category is originally defined as the category of additive sheaves on 
the site $\sp^{\mathrm{fp}}_{E}$ of compact spectra whose $E$-homology is finitely generated and projective as an $E_{*}$-module.
We prove that this has an interpretation in terms of the presentable stable envelope of an exact $\infty$-category.

\begin{theorem}[\Cref{syn:E-based-synthetic-exact}]
    The $\infty$-category $\sp^{\mathrm{fp}}_{E}$ of $E$-finite projective spectra admits an exact structure where a sequence
    $
        X \to Y \to Z
    $
    is exact if and only if the associated sequence on $E$-homology is a short exact sequence. Moreover, there is an equivalence 
    \[
        \Syn_{E} \simeq \psh_{\lex}^{\st}(\sp^{\mathrm{fp}}_{E})
    \]
    between the $\infty$-category of $E$-based synthetic spectra and the presentable stable envelope of
    this exact $\infty$-category.
\end{theorem}

\begin{rem}
    The authors believe that the relationship between ``synthetic''-type categories and exact categories is stronger than shown in this paper. This may be further explored in the future.
\end{rem}

In \cref{sec:delooping}, we define a notion of \emph{localising invariant} for exact $\infty$-categories and prove that every space-valued localising invariant deloops uniquely to a spectrum-valued localising invariant, generalising a theorem of Cisinski and Khan for localising invariants of stable $\infty$-categories \cite{ck:a1-invariance}.

In particular, this allows us to define a non-connective algebraic K-theory functor $\bbK$ which enjoys the following universal property.

\begin{prop}[\cref{cor:nc-K-universal-property}]
    Non-connective algebraic K-theory is the initial localising, spectrum-valued invariant under the (suspension spectrum of the) groupoid core.
\end{prop}

We also show that this non-connective algebraic K-theory functor coincides with the functors defined previously by Schlichting for exact 1-categories \cite{schlichting:delooping} and by Blumberg, Gepner and Tabuada for stable $\infty$-categories \cite{bgt:alg-k}.
Finally, we prove a version of the Gillet--Waldhausen theorem for non-connective algebraic K-theory of exact $\infty$-categories.

\begin{prop}[Gillet--Waldhausen, \cref{cor:gillet-waldhausen}]
    The unit transformation $\cC \to \derb(\cC)$ induces an equivalence
    \[ \bbK(\cC) \xrightarrow{\sim} \bbK(\derb(\cC)) \]
    for every exact $\infty$-category $\cC$.
\end{prop}

This also implies that $\bbK$ refines to a lax symmetric monoidal functor, see \cref{cor:k-lax-monoidal}.

\addtocontents{toc}{\SkipTocEntry}
\subsection*{Conventions}
\begin{enumerate}
    \item From this point onwards, the word ``category'' means ``$\infty$-category''.
    \item We denote by $\Spc$ the category of anima/spaces/$\infty$-groupoids/weak homotopy types.
    \item For a small category $\cC$, we denote by $\psh(\cC) := \Fun(\cC^\op,\Spc)$ the category of presheaves on $\cC$ and let $\yo \colon \cC \to \psh(\cC)$ be the Yoneda embedding.
    \item We use the notions of \emph{Waldhausen categories} and \emph{exact categories} as defined in \cite{barwick:heart}.
     In particular, we refer to the distinguished morphisms in a Waldhausen category $\cC$ as \emph{ingressive} morphisms, and denote this subcategory by $\cC_\ing$.
     If $\cC$ is exact, we refer to the cofibres of ingressive morphisms as \emph{egressive} morphisms, and denote the subcategory of these by $\cC_\egr$.
\end{enumerate}

\addtocontents{toc}{\SkipTocEntry}
\subsection*{Acknowledgements}
We are indebted to Vova Sosnilo for explaining to us the proof of \cref{prop:sosnilo} and allowing us to include it in this article.
The authors are grateful to Clark Barwick and Erlend B{\o}rve for various helpful discussions and for helpful comments.
The authors also want to thank Drew Heard for helpful comments on an earlier draft of this paper.

MN is supported by grant number TMS2020TMT02 from
the Trond Mohn Foundation.
CW is supported by the CRC 1085 ``Higher Invariants'' funded by the Deutsche Forschungsgemeinschaft (DFG).

\section{The presentable Gabriel--Quillen embedding}\label{sec:gabriel-quillen}

\begin{defn}
 Let $\cC$ be a small category with finite coproducts.
 \begin{enumerate}
  \item A functor $X \colon \cC^\op \to \cX$ is \emph{additive} if it preserves finite products.
  \item Suppose that $\cC_\ing$ is a Waldhausen structure on $\cC$. An additive functor $X \colon \cC^\op \to \cX$ is \emph{left exact} if it sends exact sequences to fibre sequences.
  \end{enumerate}
 Whenever they are defined, denote by $\psh_\Sigma(\cC)$ and $\psh_\lex(\cC)$ the full subcategories of $\psh(\cC)$ spanned by the additive and left exact presheaves, respectively.
\end{defn}

The category of additive presheaves $\psh_\Sigma(\cC)$ on a small category with finite coproducts is an accessible localisation of $\psh(\cC)$ \cite[Proposition~5.5.8.10 (1)]{HTT}.
The Yoneda embedding factors over $\psh_\Sigma(\cC)$, and the induced functor $\yo_\Sigma \colon \cC \to \psh_\Sigma(\cC)$ preserves finite coproducts \cite[Proposition~5.5.8.10 (2)]{HTT}.
The inclusion functor $\psh_\Sigma(\cC) \to \psh(\cC)$ preserves sifted colimits \cite[Proposition~5.5.8.10 (4)]{HTT}, and the category $\cP_\Sigma(\cC)$ is projectively generated and presentable \cite[Proposition~5.5.8.25]{HTT}.

\begin{defn}\label{def:comparison-map}
 Let $i \colon x \rightarrowtail y$ be an ingressive morphism in a Waldhausen category $\cC$.
 Then $i$ gives rise to a comparison map
 \[ \kappa(i) \colon \cofib(\yo_\Sigma(i)) \to \yo_\Sigma(\cofib(i)) \]
 in $\psh_\Sigma(\cC)$.
 Define the \emph{elementary acyclic} associated to $i$ as
 \[ E(i) := \cofib(\kappa(i)). \]
\end{defn}

\begin{rem}\label{rem:lex-local}
 An additive presheaf on a Waldhausen category $\cC$ is left exact if and only if it is local with respect to the collection of morphisms
 \[ \{ \kappa(i) \mid i \in \cC_\ing \}. \]
\end{rem}

\begin{prop}\label{prop:pshlex-waldhausen}
 Let $\cC$ be a small Waldhausen category.
 Then the inclusion $\psh_\lex(\cC) \subseteq \psh_\Sigma(\cC)$ admits an accessible left adjoint $\LEx \colon \psh_\Sigma(\cC) \to \psh_\lex(\cC)$ which preserves compact objects.
 In particular, $\psh_\lex(\cC)$ is a compactly generated presentable category.
 
 The Yoneda embedding of $\cC$ factors through $\psh_\lex(\cC)$, and the resulting functor $\yo_\lex \colon \cC \to \psh_\lex(\cC)$ preserves finite coproducts and sends exact sequences to cofibre sequences.
\end{prop}
\begin{proof}
 Since $\cC$ is small, \cref{rem:lex-local} and \cite[Proposition~5.5.4.15]{HTT} imply that the inclusion functor $\psh_\lex(\cC) \to \psh_\Sigma(\cC)$ admits an accessible left adjoint.
 The localisation functor $\LEx$ preserves compact objects because the inclusion functor $\psh_\lex(\cC) \subseteq \psh_\Sigma(\cC)$ preserves filtered colimits.
 By inspection, the Yoneda embedding induces a fully faithful functor $\yo_\lex \colon \cC \to \psh_\lex(\cC)$ which preserves finite coproducts and sends exact sequences to cofibre sequences.
\end{proof}

It appears that little more can be said about $\yo_\lex$ for an arbitrary Waldhausen category $\cC$.
However, the situation improves greatly when $\cC$ is exact, as we are going to discuss in the remainder of this section.

The category $\psh_\Sigma(\cC)$ is prestable if and only if $\cC$ is additive \cite[Proposition~C.1.5.7]{SAG}.
In this case, $\psh_\Sigma(\cC)$ is a \emph{Grothendieck prestable category} in the sense that it is prestable and presentable, and filtered colimits commute with finite limits in $\psh_\Sigma(\cC)$ \cite[Definition~C.1.4.2]{SAG}.
Either as a consequence of the proof of \cite[Proposition~C.1.5.7]{SAG} or by applying \cite[Corollary~2.10]{GGN:multiplicative-loop-spaces}, the forgetful functor $\Omega^\infty \colon \Sp_{\geq 0} \to \Spc$ then induces an equivalence
\[ \Fun^\oplus(\cC^\op,\Sp_{\geq 0}) \xrightarrow{\simeq} \psh_\Sigma(\cC), \]
where the domain is the full subcategory of $\Fun(\cC^\op,\Sp_{\geq 0})$ spanned by those presheaves which preserve finite products.
Note that both limits and colimits in $\Fun^\oplus(\cC^\op,\Sp_{\geq 0})$ are computed pointwise.
In the sequel, we will tacitly identify $\psh_\Sigma(\cC)$ with the category of additive $\Sp_{\geq 0}$-valued presheaves.

Another way to phrase \cref{prop:pshlex-waldhausen} for an exact category $\cC$ is to say that we have a semi-orthogonal decomposition
\[\begin{tikzcd}
	\ker(\LEx)\ar[r, shift left=2]\ar[r, phantom, "\vdash" rotate=90] & \psh_\Sigma(\cC)\ar[r, shift left=2, "\LEx"]\ar[l, shift left=2, "\Acycl"]\ar[r, phantom, "\vdash" rotate=90] & \psh_\lex(\cC)\ar[l, shift left=2, "i_\cC"]
\end{tikzcd}\]
where
\[ \Acycl(X) \simeq \fib(X \to i_\cC\LEx(X)) \]
is the fibre of the unit transformation.
Since filtered colimits preserve finite limits in $\Sp_{\geq 0}$ and $i_\cC$ preserves filtered colimits, the same is true for $\Acycl$.
In the following, we will typically suppress explicit mention of the inclusion functors.

\begin{defn}
 Let $\cC$ be a category and let $S$ be a subcategory of $\cC$.
 \begin{enumerate}
   \item A functor $X \colon \cC^\op \to \Ab$ is \emph{weakly $S$-effaceable} if for every $x \in \cC$ and $\xi \in X(x)$ there exists some morphism $s \colon x' \to x$ in $S$ such that
   \[ s^*(\xi) = 0 \in X(x'). \]
 	\item A presheaf $X \colon \cC^\op \to \Sp_{\geq 0}$ is \emph{weakly $S$-effaceable} if $\pi_kX$ is weakly $S$-effaceable for every $k \in \bbN$. 
 \end{enumerate}
\end{defn}

This notion allows us to give a concise description of the objects $E(i)$.

\begin{lem}\label{lem:elementary-acyclics}
 Let $\cC$ be an exact category and let $X \colon \cC^\op \to \Sp_{\geq 0}$ be an additive presheaf.
 The following are equivalent:
 \begin{enumerate}
  \item $X$ is an elementary acyclic associated to some ingressive morphism;
  \item there exists an egressive morphism $p \colon y \twoheadrightarrow z$ such that
   \[ X \simeq \pi_0\cofib(\yo_\Sigma(p)) \simeq \coker( \pi_0\yo_\Sigma(y) \to \pi_0\yo_\Sigma(z)); \]
  \item there exists a morphism $f \colon y \to z$ in $\cC$ such that
   \[ X \simeq \pi_0\cofib(\yo_\Sigma(f)) \simeq \coker( \pi_0\yo_\Sigma(y) \to \pi_0\yo_\Sigma(z)) \]
   and $X$ is weakly $\cC_\egr$-effaceable.
 \end{enumerate}
\end{lem}
\begin{proof}
 See \cite[Proposition~3.7]{sw:exact-stable} for a proof.
\end{proof}

\begin{theorem}\label{thm:plex}
 Let $\cC$ be a small exact category.
 \begin{enumerate}
  \item\label{thm:plex-1} The localisation functor $\LEx \colon \psh_\Sigma(\cC) \to \psh_\lex(\cC)$ preserves finite limits and $\psh_\lex(\cC)$ is Grothendieck prestable.
  \item\label{thm:plex-2} Every discrete, weakly $\cC_\egr$-effaceable additive presheaf lies in $\ker(\LEx)$.
  \item\label{thm:plex-3} Every object in $\ker(\LEx)$ is weakly $\cC_\egr$-effaceable.
  \item\label{thm:plex-4} The collection $\{ E(i) \}_{i \in \cC_\ing}$ of elementary acyclics generates $\ker(\LEx)$ under colimits and extensions.
   \item\label{thm:plex-5} For every pointed, cocomplete category $\cD$, the restriction functor
   \[ \yo_\lex^* \colon \Fun^{\Ldec}(\psh_\lex(\cC),\cD) \to \Funex{}(\cC,\cD) \]
   is an equivalence, where $\cD$ is equipped with the maximal Waldhausen structure.
   Moreover, a functor $\cC \to \cD$ is exact if and only if it preserves finite coproducts and sends exact sequences to cofibre sequences.
 \end{enumerate}
\end{theorem}
\begin{proof}
 Since $\psh_\Sigma(\cC)$ is Grothendieck prestable, \cite[Proposition~C.2.3.1]{SAG} implies that $\psh_\lex(\cC)$ is Grothendieck prestable if $\LEx$ preserves finite limits.
 For assertion~\eqref{thm:plex-1}, it is therefore enough to show that $\LEx$ preserves finite limits.
 
 Egressives in $\cC$ are closed under pullbacks, so we obtain a quasi-topology on $\cC$ in which a covering of an object $x$ is given precisely by an egressive morphism $x' \twoheadrightarrow x$.
 The corresponding sheafification functor $\psh(\cC) \to \sh(\cC)$ preserves finite limits \cite[Proposition~6.2.2.7 \& Corollary~6.2.1.6]{HTT} and restricts to a left exact localisation functor $\psh_\Sigma(\cC) \to \sh_\Sigma(\cC)$ onto the category of additive sheaves by \cite[Corollary~2.7]{pstragowski:synthetic}.

 We wish to identify the full subcategory $\sh_\Sigma(\cC)$ of additive sheaves with the full subcategory of left exact presheaves.
 To do so, we use the equivalence $\psh_\Sigma(\cC) \simeq \Fun^\oplus(\cC^\op,\Sp_{\geq 0})$.
 For an exact sequence $x \overset{i}{\rightarrowtail} y \overset{p}{\twoheadrightarrow} z$ in $\cC$, we obtain a commutative diagram
 \[\begin{tikzcd}
	\yo_\Sigma(x)\ar[r, "\yo_\Sigma(i)"]\ar[d, "\id"] & \yo_\Sigma(y)\ar[r]\ar[d, "\id"] & \cofib(\yo_\Sigma(i))\ar[d, "\kappa(i)"] \\
	\yo_\Sigma(x)\ar[r, "\yo_\Sigma(i)"] & \yo_\Sigma(y)\ar[r, "\yo_\Sigma(p)"] & \yo_\Sigma(z)
 \end{tikzcd}\]
 The upper row is (pointwise) a cofibre sequence of connective spectra, whereas the lower row is (pointwise) a fibre sequence of connective spectra.
 It follows that $\kappa(i)$ is $0$-coconnected.
 Since $\yo_\Sigma(y) \to \cofib(\yo_\Sigma(i))$ is a $\pi_0$-surjection, the induced map on $\pi_0$ identifies $\pi_0(\cofib(\yo_\Sigma(i))$ with the sub-presheaf of $\pi_0\yo_\Sigma(z)$ given by those morphisms which factor over $p$.
 Hence $E(i) = \cofib(\kappa(i))$ is discrete, and $\kappa(i)$ itself is precisely the sieve associated to the covering $p \colon y \twoheadrightarrow z$.
 It follows that an additive presheaf on $\cC$ is a sheaf if and only if it is left exact, so the restricted sheafification functor is precisely $\LEx \colon \psh_\Sigma(\cC) \to \psh_\lex(\cC)$.
 This proves assertion~\eqref{thm:plex-1}.
 Alternatively, see \cite[Theorem~2.8]{pstragowski:synthetic} for a proof in terms of \v{C}ech covers.

Since sheafification involves only limits and filtered colimits, and $\Ab \subseteq \Sp_{\geq 0}$ is closed under both of these operations, the adjunction $\LEx \colon \psh_\Sigma(\cC) \rightleftarrows \psh_\lex(\cC) \cocolon i_\cC$ restricts to an adjunction
 \[ \LEx^\heartsuit \colon \Fun^\oplus(\cC^\op,\Ab) \rightleftarrows \Fun^\lex(\cC^\op,\Ab) \cocolon \inc. \]
 It is well-known that the kernel of $\LEx^\heartsuit$ is given precisely by the full subcategory of weakly $\cC_\egr$-effaceable presheaves \cite[Lemma~A.7.11]{TT}; in fact, this follows directly by unwinding the explicit formula for the sheafification functor.
 This proves assertion~\eqref{thm:plex-2}.

 By \cite[Proposition~C.2.3.8]{SAG}, the kernel of $\LEx$ is a localising subcategory in the sense of \cite[Definition~C.2.3.3]{SAG}.
 If $X$ is now an object in $\ker(\LEx)$, then $\LEx^\heartsuit \pi_0 X \simeq 0$ as well, so $\pi_0 X$ is weakly $\cC_\egr$-effaceable.
 Since $\ker(\LEx)$ is a localising subcategory, we conclude from the cofibre sequence $\tau_{\geq 1} X \to X \to \pi_0X$ that $\tau_{\geq 1}X$ also lies in the kernel of $\LEx$.
 Since $\tau_{\geq 1}X$ is $1$-connective, there exists a cofibre sequence $\Sigma^{-1} \tau_{\geq 1} X \to 0 \to \tau_{\geq 1} X$, and it follows that $\Sigma^{-1}\tau_{\geq 1}X \in \ker(\LEx)$.
 Continuing by induction, we conclude that $\pi_kX$ is weakly $\cC_\egr$-effaceable for every $k \geq 0$, so assertion~\eqref{thm:plex-3} holds.

 We proceed to prove \eqref{thm:plex-4}.
 Let $X \in \ker(\LEx)$.
Like every additive presheaf, $X$ admits a filtration
\[ X_0 \to X_1 \to \ldots \to X \]
such that $\colim_n X_n \simeq X$ and $\cofib(X_n \to X_{n+1}) \simeq \bigoplus_\alpha \Sigma^n\yo_\Sigma(x_\alpha)$.
Since $\Acycl$ commutes with filtered colimits, we have $\Acycl(X) \simeq \colim_n \Acycl(X_n)$.

We claim that $\Acycl(X_n)$ is $(n-1)$-truncated.
First, note that $\Acycl(\bigoplus_\alpha \yo_\Sigma(x_\alpha)) \simeq 0$ because $\Acycl$ preserves filtered colimits and all representables are left exact.
For $n > 0$, we obtain a commutative diagram of additive presheaves
\[\begin{tikzcd}[column sep=1.5em, row sep=1.5em]
	\Acycl(\bigoplus_\alpha \Sigma^{n-1}\yo_\Sigma(x_\alpha))\ar[r]\ar[d] & 0\ar[r]\ar[d] & \Acycl(\bigoplus_\alpha \Sigma^{n}\yo_\Sigma(x_\alpha))\ar[d] \\
	\bigoplus_\alpha \Sigma^{n-1}\yo_\Sigma(x_\alpha)\ar[r]\ar[d] & 0\ar[r]\ar[d] & \bigoplus_\alpha \Sigma^{n}\yo_\Sigma(x_\alpha)\ar[d] \\
	\LEx(\bigoplus_\alpha \Sigma^{n-1}\yo_\Sigma(x_\alpha))\ar[r] & 0\ar[r] & \LEx(\bigoplus_\alpha \Sigma^{n}\yo_\Sigma(x_\alpha))
\end{tikzcd}\]
in which all rows and columns are fibre sequences due to $\psh_\Sigma(\cC)$ and $\psh_\lex(\cC)$ being prestable.
By induction, the top left term is $(n-2)$-truncated, so the top right term is $(n-1)$-truncated.

Since $X_0$ is a sum of representables, this shows that $\Acycl(X_0)$ is trivial.
The chosen filtration on $X$ induces fibre sequences $\Acycl(X_n) \to \Acycl(X_{n+1}) \to \Acycl(\bigoplus_\alpha \Sigma^{n+1}\yo_\Sigma(x_\alpha))$.
The right hand term is $n$-truncated, and the left hand term is $(n-1)$-truncated by induction, so $\Acycl(X_{n+1})$ is $n$-truncated as claimed.

Thus, it suffices to show that every $n$-truncated object in $\ker(\LEx)$ can be obtained from the elementary acyclics through colimits and extensions.
This will be shown by an induction on $n$.
For $n = 0$, we have to consider a discrete and weakly $\cC_\egr$-effaceable presheaf $A$.
Choose a $\pi_0$-surjection $f \colon \bigoplus_\alpha \yo_\Sigma(x_\alpha) \to A$.
Since $A$ is weakly effaceable, there exist egressives $(p_\alpha)_\alpha$ such that $f$ factors through the cofibre of $\bigoplus_\alpha \yo_\Sigma(p_\alpha)$.
The induced map further factors through $\bigoplus_\alpha \pi_0\cofib(\yo_\Sigma(p_\alpha))$ by the discreteness of $A$.
Since $\pi_0\cofib(\yo_\Sigma(p_\alpha)) \simeq E(i_\alpha)$ for some $i_\alpha$ by \cref{lem:elementary-acyclics}, this shows that every discrete, weakly effaceable object admits a resolution by objects which are sums of elementary acyclics.
In particular, the subcategory of discrete objects in $\ker(\LEx)$ is generated by the elementary acyclics under colimits.

For an arbitrary $n$-truncated object $A \in \ker(\LEx)$, there exists an extension $\tau_{\geq 1}A \to A \to \pi_0(A)$.
Since $\tau_{\geq 1}A$ is a suspension of an $(n-1)$-truncated object, it follows by induction that $\tau_{\geq 1}A$ lies in the subcategory generated by the elementary acyclics.
We have just seen that the same is true for $\pi_0(A)$, so $A$ itself is also generated by the elementary acyclics under colimits and extensions.

 We are left with proving the universal property claimed in \eqref{thm:plex-5}.
 By \cite[Proposition~5.5.4.20]{HTT}, restriction along $\LEx$ induces an equivalence
 \[ \LEx^* \colon \Fun^{\Ldec}(\psh_\lex(\cC),\cD) \xrightarrow{\sim} \Fun^{\Ldec}_{K}(\psh_\Sigma(\cC),\cD), \]
 where $K := \{ \kappa(i) \mid i \in \cC_\ing \}$ denotes the collection of comparison maps introduced in \cref{def:comparison-map}.

 Since a colimit-preserving functor $f \colon \psh_\Sigma(\cC) \to \cD$ inverts every morphism $\kappa(i)$ if and only if $f \circ \yo_\Sigma$ sends exact sequences to cofibre sequences, we obtain a pullback square 
 \[\begin{tikzcd}
  \Fun^{\colim}_{K}(\psh_\Sigma(\cC),\cD)\ar[r]\ar[d] &  \Fun^{\sqcup,\mathrm{rex}}(\cC,\cD)\ar[d] \\
  \Fun^{\colim}_{}(\psh_\Sigma(\cC),\cD)\ar[r] & \Fun^\sqcup(\cC,\cD)
 \end{tikzcd}\]
 in which the top right corner denotes the category of functors preserving finite coproducts and sending exact sequences to cofibre sequences.
 The bottom horizontal arrow is an equivalence by \cite[Proposition~5.5.8.15]{HTT}, so the top horizontal arrow is also an equivalence.
 Thus, we obtain an equivalence
 \[ \Fun^{\Ldec}(\psh_\lex(\cC),\cD) \xrightarrow{\sim} \Fun^{\sqcup,\mathrm{rex}}(\cC,\cD). \]
 In particular, if $f \colon \cC \to \cD$ preserves finite coproducts and sends exact sequences to cofibre sequences, it extends to a colimit-preserving functor $F \colon \psh_\lex(\cC) \to \cD$ with $F \circ \yo_\lex \simeq f$.
 Since $\yo_\lex$ is an exact functor by construction, it follows that $f$ sends all exact squares in $\cC$ to pushout squares in $\cD$, so $f$ is already an exact functor of Waldhausen categories.
 This shows that the inclusion $\Funex{}(\cC,\cD) \to \Fun^{\sqcup,\mathrm{rex}}(\cC,\cD)$ is an equivalence, finishing the proof.
 \end{proof}

\begin{rem}
 In general, $\ker(\LEx)$ is properly contained in the subcategory of all weakly $\cC_\egr$-effaceable presheaves.
 For example, if every morphism in $\cC$ is egressive, then every presheaf is weakly $\cC_\egr$-effaceable since one can restrict along $0 \to x$.
\end{rem}

As direct consequences of \cref{thm:plex}, we can prove a presentable version of the Gabriel--Quillen embedding theorem and recover Klemenc's stable envelope of an exact category.

\begin{lem}\label{lem:pshlex-special}
 Let $\cC$ be a small exact category and let
 \[ X \xrightarrow{f} Y \xrightarrow{q} \yo_\lex(z) \]
 be a cofibre sequence in $\psh_\lex(\cC)$ with $z \in \cC$.
 Then there exists a commutative diagram
 \[\begin{tikzcd}
	\yo_\lex(y)\ar[dr, "\yo_\lex(p)"]\ar[d] & \\
	Y\ar[r, "q"] & \yo_\lex(z)
 \end{tikzcd}\]
 such that $p \colon y \twoheadrightarrow z$ is an egressive morphism in $\cC$.
\end{lem}
\begin{proof}
	Denote by $C$ the cofibre of $q$ in $\psh_\Sigma(\cC)$.
    The pointwise suspension $\Sigma \circ X$ acquires a canonical map $c \colon \Sigma \circ X \to C$ which becomes an equivalence upon application of $\LEx$.
    Hence $\cofib(c)$ is weakly $\cC_\egr$-effaceable by \cref{thm:plex}~\eqref{thm:plex-3}.
	In particular, the morphism $\yo(z) = \yo_\lex(z) \to \cofib(c)$ represents an element in $\pi_0(\cofib(c))(z)$, and this element becomes trivial upon precomposition with some egressive $p \colon y \twoheadrightarrow z$.
	Hence $\yo_\lex(p)$ lifts to a morphism $\yo_\lex(y) \to Y$.
\end{proof}

To establish the Gabriel--Quillen embedding theorem, we will also require Quillen's `obscure axiom'.

\begin{lem}\label{lem:obscure-axiom}
    Let $\cC$ be an exact category.
    Suppose that $p \colon x \to y$ is a morphism in $\cC$ which admits a fibre and that there exists a morphism $f \colon x' \to x$ such that $pf$ is egressive in $\cC$.
    Then $p$ is egressive.
\end{lem}
\begin{proof}
    See \cite[Lemma~2.5]{sw:exact-stable}.
\end{proof}

\begin{theorem}[Gabriel--Quillen embedding theorem]\label{thm:gabriel-quillen}
 Let $\cC$ be a small Waldhausen category. The following are equivalent:
 \begin{enumerate}
  \item\label{it:gabriel-quillen1} $\cC$ is exact.
  \item \label{it:gabriel-quillen2} The category of left exact presheaves $\psh_\lex(\cC)$ is prestable and the Yoneda embedding $\yo_\lex \colon \cC \to \psh_\lex(\cC)$ identifies $\cC$ with an extension-closed subcategory.
  \item \label{it:gabriel-quillen3} There exist a prestable category $\cD$ and a fully faithful functor $\cC \to \cD$ which identifies $\cC$ with an extension-closed subcategory of $\cD$.
 \end{enumerate}
\end{theorem}
\begin{proof}
	Evidently, \eqref{it:gabriel-quillen2} implies \eqref{it:gabriel-quillen3}, and \eqref{it:gabriel-quillen3} implies \eqref{it:gabriel-quillen1} because extension-closed subcategories of exact categories are exact.
	
	Suppose that \eqref{it:gabriel-quillen1} holds. Then \cref{thm:plex} shows that $\yo_\lex \colon \cC \to \psh_\lex(\cC)$ is an exact embedding into a Grothendieck prestable category.
	We are left with showing that the essential image of $\yo_\lex$ is closed under extensions and that $\yo_\lex \colon \cC \to \psh_\lex(\cC)$ reflects exact sequences.
	
	Consider a cofibre sequence $\yo_\lex(x) \to Y \to \yo_\lex(z)$ with $x, z \in \cC$. By \cref{lem:pshlex-special}, there exists an exact sequence
	\[ x' \rightarrowtail y \overset{p}{\twoheadrightarrow} z \]
	such that $p$ factors over $Y$. Since $\psh_\lex(\cC)$ is prestable, the resulting commutative square
	\[\begin{tikzcd}
	 \yo_\lex(x')\ar[r]\ar[d] & \yo_\lex(y)\ar[d] \\
	 \yo_\lex(x)\ar[r] & Y
   \end{tikzcd}\]
   is a pushout, so $Y \simeq \yo_\lex(x \cup_{x'} y)$ by exactness of $\yo_\lex$.
   
   Finally, consider a cofibre sequence $\yo_\lex(x_0) \to \yo_\lex(x_1) \to \yo_\lex(x_2)$ in $\psh_\lex(\cC)$.
   Applying \cref{lem:pshlex-special} once more, we find an egressive morphism $y \twoheadrightarrow x_2$ in $\cC$ which factors over $x_1$. Since $\psh_\lex(\cC)$ is prestable, $\yo_\lex(x_0) \to \yo_\lex(x_1) \to \yo_\lex(x_2)$ is a (pointwise) fibre sequence, so $x_1 \to x_2$ admits a fibre.
    Quillen's obscure axiom \cref{lem:obscure-axiom} implies that $x_1 \to x_2$ is an egressive morphism in $\cC$.
    We conclude that $\yo_\lex$ reflects exact sequences.
\end{proof}

For the sake of completeness, we also record the following property of the Gabriel--Quillen embedding.

\begin{prop}
    Let $\cC$ be a small exact category.
    Then $\cC$ is closed under fibres of egressives in $\psh_\lex(\cC)$ if and only if $\cC$ is weakly idempotent complete.
\end{prop}
\begin{proof}
 The proof given in \cite[Theorem~1.4]{sw:exact-stable} carries over to $\psh_\lex(\cC)$.
\end{proof}

\begin{defn}
    For any exact category $\cC$, define the \emph{presentable stable envelope} of $\cC$ to be
    \[ \psh_\lex^\st(\cC) := \Sp(\psh_\lex(\cC)). \]
\end{defn}

\begin{cor}\label{cor:presentable-stable-hull}
    For every small exact category $\cC$ and every presentable, stable category $\cD$, the restriction functor
    \[ \Fun^{\Ldec}(\psh_\lex^\st(\cC),\cD) \to \Funex{}(\cC,\cD) \]
    is an equivalence.
\end{cor}
\begin{proof}
    This is immediate by combining \cite[Corollary~1.4.4.5]{HA} with the universal property of $\psh_\lex(\cC)$ from \cref{thm:plex}~\eqref{thm:plex-5}.
\end{proof}

\begin{rem}
    Let $\cA$ be a small abelian category.
    Then $\Ind(\cA)$ is a Grothendieck abelian category (ie it is abelian, presentable, and filtered colimits are exact), and therefore gives rise to the unbounded derived category $\cD(\Ind(\cA))$ as defined in \cite[Definition~1.3.5.8]{HA}.
    By \cref{cor:presentable-stable-hull} and \cite[Proposition~1.3.5.9 \& 1.3.5.21]{HA}, the exact functor $\cA \to \Ind(\cA) \to \cD(\Ind(\cA))$ induces a colimit-preserving functor
    \[ \psh_\lex^\st(\cA) \to \cD(\Ind(\cA)). \]
    Interpreting $\psh_\lex^\st(\cA)$ as the category of spectrum-valued sheaves on $\cA$ with respect to the topology generated by the epimorphisms in $\cA$, \cite[Theorem~2.64]{pstragowski:synthetic} identifies this functor as the corresponding hypersheafification functor.

    We are not aware of a similar interpretation of the hypersheafification functor in the case of arbitrary exact categories.
\end{rem}

The property of $\yo_\lex$ formulated in \cref{lem:pshlex-special} deserves some attention on its own.
Our terminology is borrowed from \cite{schlichting:delooping}.

\begin{defn}\label{def:special}
 Let $\cC$ be an exact category.
 An extension-closed subcategory $\cU \subseteq \cC$ is \emph{left special} if the following holds:
 for every egressive morphism $p \colon x \twoheadrightarrow u$ with $x \in \cC$ and $u \in \cU$, there exists a morphism $f \colon v \to x$ with $v \in \cU$ such that the composite $pf \colon v \to u$ is an egressive morphism in $\cU$.
 
 A \emph{right special} subcategory is an extension-closed subcategory $\cU \subseteq \cC$ such that $\cU^\op \subseteq \cC^\op$ is left special.
\end{defn}

\begin{ex}\ 
	\begin{enumerate}
	 \item Let $\cC$ be an additive category with its split-exact structure and let $\cU$ be a full additive subcategory.
     Then $\cU \subseteq \cC$ is left special and right special.
	 \item \cref{thm:gabriel-quillen} and \cref{lem:pshlex-special} show that the Yoneda embedding $\yo_\lex \colon \cC \to \psh_\lex(\cC)$ exhibits $\cC$ as a left special subcategory.
     \item The full subcategory of torsion groups in the abelian category of finitely generated abelian groups is right special.
	\end{enumerate}
\end{ex}

A convenient reformulation of \cref{def:special} is the following.

\begin{lem}
 For $\cC$ exact and $\cU \subseteq \cC$ an extension-closed subcategory, the following are equivalent:
 \begin{enumerate}
 	\item $\cU$ is left special in $\cC$.
 	\item Every ingressive morphism $i \colon x \rightarrowtail y$ of $\cC$ with $\cofib(i) \in \cU$ is the pushout of an ingressive morphism in $\cU$.
 \end{enumerate}
\end{lem}
\begin{proof}
 If $\cU$ is left special, we obtain for every ingressive $i \colon x \rightarrowtail y$ a map of exact sequences
 \[\begin{tikzcd}
	u\ar[r, rightarrowtail]\ar[d] & v\ar[r, twoheadrightarrow]\ar[d] & \cofib(i)\ar[d, "\id"] \\
	x\ar[r, "i", rightarrowtail] & y\ar[r, twoheadrightarrow] & \cofib(i)
 \end{tikzcd}\]
 Hence the left square is a pushout. Conversely, the fibre of an egressive $p \colon y \to w$ with $w \in \cU$ is an ingressive with cofibre in $\cU$, so we obtain an analogous pushout square in which the right square witnesses that $\cU \subseteq \cC$ is left special.
\end{proof}

For ordinary exact categories, Keller has shown that inclusions of left special subcategories induce fully faithful functors on derived categories \cite[Section~12]{keller:derived-cats}.
On the level of small categories, \cite[Theorem~1.2]{sw:exact-stable} provides an analogous assertion for stable envelopes of exact categories.
By considering presentable categories, the meaning of this condition can be described a bit more explicitly.
Parts of the following proposition can also be found in \cite[Proposition~2.11 \& Corollary~2.44]{pstragowski:synthetic}.

\begin{prop}\label{prop:special-fullyfaithful}
 	Let $\cC$ be a small exact category and let $\cU \subseteq \cC$ be an extension-closed subcategory.
 	Denote the inclusion functor by $j \colon \cU \to \cC$.
	The following are equivalent:
	\begin{enumerate}
	 \item\label{it:special-fullyfaithful1} $\cU$ is left special in $\cC$;
	 \item\label{it:special-fullyfaithful2} the restriction functor $j^* \colon \psh_\Sigma(\cC) \to \psh_\Sigma(\cU)$ sends weakly $\cC_\egr$-effaceable presheaves to weakly $\cU_\egr$-effaceable presheaves;
	 \item\label{it:special-fullyfaithful3} the square
	  \[\begin{tikzcd}
	  \psh_\Sigma(\cC)\ar[r, "j^*"] & \psh_\Sigma(\cU) \\
	  \psh_\lex(\cC)\ar[r, "j^*"]\ar[u, "i_\cC"] & \psh_\lex(\cU)\ar[u, "i_\cU"]
	\end{tikzcd}\]
	is (vertically) left adjointable, ie the base change transformation
	\[ \LEx_\cU j^* \to j^* \LEx_\cC \]
	is an equivalence;
	\item\label{it:special-fullyfaithful4} $j^* \colon \psh_\lex(\cC) \to \psh_\lex(\cU)$ preserves colimits.
	\end{enumerate}	
 	If these conditions are satisfied, the left adjoint $j_\natural \colon \psh_\lex(\cU) \to \psh_\lex(\cC)$ of $j^*$ is fully faithful.
\end{prop}
\begin{proof}
	To show that \eqref{it:special-fullyfaithful1} implies \eqref{it:special-fullyfaithful2}, assume that $\cU$ is left special in $\cC$ and consider a weakly $\cC_\egr$-effaceable presheaf $X \in \psh_\Sigma(\cC)$.
	By definition, there exists for every $u \in \cU$ and $\alpha \in \pi_kX(u)$ an egressive morphism $p \colon x \twoheadrightarrow u$ such that $p^*\alpha = 0$.
	Since $\cU$ is left special, there also exists a morphism $f \colon v \to x$ such that the composite $pf$ is an egressive morphism in $\cU$, and $(pf)^*\alpha = 0$.
	
	Assume \eqref{it:special-fullyfaithful2}.
	Then $j^*$ clearly sends discrete weakly $\cC_\egr$-effaceable presheaves to discrete weakly $\cU_\egr$-effaceable presheaves.
    Since $j^* \colon \psh_\Sigma(\cC) \to \psh_\Sigma(\cU)$ preserves all colimits, it restricts to the kernels of the sheafification functors by \cref{thm:plex}~\eqref{thm:plex-4}.
	The base change transformation in \eqref{it:special-fullyfaithful3} is given by the composite
	\[ \LEx_\cU j^* \to \LEx_\cU j^* i_\cC \LEx_\cC \simeq \LEx_\cU i_\cU j^* \LEx_\cC \xrightarrow{\sim} j^* \LEx_\cC, \]
	where the first and last arrow are induced by the respective unit and counit.
	Therefore, the base change tranformation is an equivalence if the sheafification of the unit morphism $j^*(u) \colon j^* \to j^* i_\cC \LEx_\cC$ is an equivalence.
	This is the case precisely if the cofibre of this morphism lies in $\ker(\LEx_\cU)$.
	As $j^* \colon \psh_\Sigma(\cC) \to \psh_\Sigma(\cU)$ preserves cofibre sequences, it follows that $\LEx_\cU j^*(u)$ is an equivalence, so \eqref{it:special-fullyfaithful3} holds.

	To see that \eqref{it:special-fullyfaithful3} implies \eqref{it:special-fullyfaithful4},
	observe that for every diagram $X \colon I \to \psh_\lex(\cC)$ we have
	\begin{align*}
	\colim_I j^*X
	 &\simeq \LEx_\cU \colim_I i_\cU j^*X \simeq \LEx_\cU \colim_I j^* i_\cC X \\
	 &\simeq \LEx_\cU j^* \colim_I i_\cC X \simeq j^* \LEx_\cC \colim_I i_\cC X \\
	 &\simeq j^*(\colim_I X).
	\end{align*}
	Suppose that \eqref{it:special-fullyfaithful4} holds and consider an egressive morphism $p \colon x \twoheadrightarrow u$ with $x \in \cC$ and $u \in \cU$. Then
	\[ j^*\yo_\lex(\fib(p)) \to j^*\yo_\lex(x) \xrightarrow{\yo(p)} j^*\yo_\lex(u) \]
	is a cofibre sequence in $\psh_\lex(\cU)$.
	Since $j^*\yo_\lex(u)$ is representable and $\cU \subseteq \psh_\lex(\cU)$ is left special by \cref{lem:pshlex-special}, there exists a morphism $f \colon v \to x$ such that the composite $pf$ is egressive in $\cU$.
	So $\cU$ is left special in $\cC$.
	
	To conclude that $j_\natural$ is fully faithful under these conditions, note that $j_\natural \simeq \LEx_\cC j_! i_\cU$, where $j_! \colon \psh_\Sigma(\cU) \to \psh_\Sigma(\cC)$ is the left adjoint to $j^*$.
	The associated unit transformation is given by the composite
	\[ \id \xrightarrow{\sim} \LEx_\cU i_\cU \xrightarrow{\sim} \LEx_\cU j^* j_! i_\cU \to j^* \LEx_\cC j_! i_\cU \simeq j^* j_\natural, \]
  the third arrow being the base change transformation.
\end{proof}

\begin{cor}\label{cor:special-recollement}
	Let $\cC$ be a small exact category and let $\cU \subseteq \cC$ be a left special subcategory.
	Then there exists a recollement
	\[\begin{tikzcd}
	 \psh^\cU_\lex(\cC)\ar[r, hookrightarrow] & \psh_\lex(\cC)\ar[l, bend right=60]\ar[l, bend right=25, phantom, "\dashv" rotate=-90]\ar[l, bend left=25, phantom, "\dashv"' rotate =-90]\ar[l, bend left=60]\ar[r, "j^*"] & \psh_\lex(\cU)\ar[l, bend right=60, "j_\natural"'] \ar[l, bend right=35, phantom, "\dashv" rotate=-90]\ar[l, bend left=25, phantom, "\dashv" rotate=-90]\ar[l, bend left=60, "j_*"]
	\end{tikzcd}\]
	of Grothendieck prestable categories, where $\psh_\lex^\cU(\cC)$ denotes the full subcategory of left exact presheaves vanishing on $\cU$.
\end{cor}
\begin{proof}
 Since $\cU$ is left special in $\cC$, \cref{prop:special-fullyfaithful} implies that $j^*$ preserves colimits, so it admits a right adjoint $j_* \colon \psh_\lex(\cU) \to \psh_\lex(\cC)$.
 It follows formally that the inclusion of $\psh_\lex^\cU(\cC)$ admits both adjoints, and that these are given by
 \[ \cofib(j_\natural j^* \Rightarrow \id)\quad\text{and}\quad \fib(\id \Rightarrow j_*j^*), \]
 respectively.
 In particular, $\psh_\lex^\cU(\cC)$ is closed under limits and colimits, so it is Grothendieck prestable.
\end{proof}

It is easy to recover Klemenc's stable envelope of an exact category from the presentable Gabriel--Quillen embedding.

\begin{defn}
 For a small exact category $\cC$, denote by $\derb_{\geq 0}(\cC)$ the smallest full prestable subcategory of $\psh_\lex(\cC)$ containing the essential image of the Yoneda embedding $\yo_\lex$.
 Define the \emph{stable envelope}
 \[ \derb(\cC) := \stab(\derb_{\geq 0}(\cC)) \]
 as the Spanier--Whitehead stabilisation of $\derb_{\geq 0}(\cC)$.
\end{defn}

\begin{rem}\label{rem:derb-special}
    \cref{prop:special-fullyfaithful} directly implies \cite[Theorem~1.2]{sw:exact-stable}: the functor $\derb(\cU) \to \derb(\cC)$ is fully faithful for every left or right special subcategory $\cU \subseteq \cC$.
\end{rem}

\begin{prop}\label{prop:derb-pos-left-adjoint}
    Let $\cC$ be a small exact category and let $\cD$ be a prestable category.
    Restriction along $\cC \to \derb_{\geq 0}(\cC)$ induces an equivalence
    \[ \Fun^\exct(\derb_{\geq 0}(\cC),\cD) \xrightarrow{\sim} \Fun^\exct(\cC,\cD). \]
    In particular, $\derb_{\geq 0}$ refines to a left adjoint functor to the fully faithful inclusion of small prestable categories to small exact categories.
\end{prop}
\begin{proof}
 For any small exact category $\cE$, the image of an exact functor $\cE \to \cD$ is contained in a small prestable subcategory of $\cD$, for example the closure of its image under finite colimits and extensions.
 Therefore, we have an equivalence
 \[ \colim_{\substack{\cD' \subseteq \cD \\ \text{small prestable subcat.}}} \Funex{}(\cE,\cD') \xrightarrow{\sim} \Funex{}(\cE,\cD) \]
 which is natural in $\cE$.
 This allows us to assume without loss of generality that $\cD$ is small.

 Since $\derb_{\geq 0}(\cC)$ is generated under finite colimits and extensions by $\cC$, the square
    \[\begin{tikzcd}
        \Fun^\exct(\derb_{\geq 0}(\cC),\cD)\ar[r]\ar[d] & \Fun^\exct(\cC,\cD)\ar[d] \\
        \Fun^\exct(\derb_{\geq 0}(\cC),\psh_\lex(\cD))\ar[r] & \Fun^\exct(\cC,\psh_\lex(\cD))
    \end{tikzcd}\]
    is a pullback.
    Since $\cC \to \derb_{\geq 0}(\cC)$ is left special, the induced colimit-preserving functor $\psh_\lex(\cC) \to \psh_\lex(\derb_{\geq 0}(\cC))$ is fully faithful, and it is essentially surjective since $\derb_{\geq 0}(\cC)$ itself is generated by $\cC$.
    It follows from \cref{thm:plex} that the lower horizontal arrow in the above diagram is an equivalence.
\end{proof}

\begin{cor}\label{cor:derb-left-adjoint}
 Let $\cC$ be a small exact category and let $\cD$ be a stable category.
 Restriction along $\cC \to \derb(\cC)$ induces an equivalence
 \[ \Fun^\exct(\derb(\cC),\cD) \xrightarrow{\sim} \Fun^\exct(\cC,\cD). \]
 In particular, $\derb$ refines to a left adjoint of the fully faithful inclusion of small stable categories to small exact categories.
\end{cor}
\begin{proof}
    This is immediate from \cref{prop:derb-pos-left-adjoint} because $\stab$ is in particular left adjoint to the inclusion of stable categories into prestable categories.
\end{proof}

\begin{cor}\label{cor:ind-derb-is-plex}
    The inclusions $\derb_{\geq 0} \to \psh_\lex$ and $\derb \to \psh_\lex^\st$ induce equivalences
    \[ \Ind \circ \derb_{\geq 0} \simeq \psh_\lex \quad\text{and}\quad \Ind \circ \derb \simeq \psh_\lex^\st. \]
\end{cor}
\begin{proof}
    For every small exact category $\cC$ and every Grothendieck prestable category $\cD$, we have equivalences
    \[ \Fun^{\Ldec}(\Ind(\derb_{\geq 0}(\cC)),\cD) \xrightarrow{\sim} \Funex{}(\derb_{\geq 0}(\cC),\cD) \xrightarrow{\sim} \Funex{}(\cC,\cD) \]
    by \cref{prop:derb-pos-left-adjoint}.
    Since $\Ind(\derb_{\geq 0}(\cC))$ is Grothendieck prestable, it satisfies the same universal property as $\psh_\lex(\cC)$ by \cref{thm:plex}~\eqref{thm:plex-5}.

    The second equivalence follows by stabilising.
\end{proof}

Finally, the presentable Gabriel--Quillen embedding also allows for an easy construction of idempotent completions of exact categories as exact categories.

\begin{lem}\label{lem:idem}
    Let $\cC$ be a small exact category.
    \begin{enumerate}
        \item The Gabriel--Quillen embedding $\yo_\lex(\cC) \colon \cC \to \psh_\lex(\cC)$ induces a fully faithful functor $\idem{\cC} \to \psh_\lex(\cC)$ which exhibits the idempotent completion as an extension-closed subcategory.
        \item In the resulting exact structure on $\idem{\cC}$, the exact sequences are precisely the retracts of exact sequences in $\cC$.
        \item\label{lem:idem-3} If $\cU$ is an extension-closed subcategory of $\cC$, then $\idem{\cU}$ is an extension-closed subcategory of $\idem{\cC}$.
    \end{enumerate} 
\end{lem}
\begin{proof}
    The idempotent completion of $\cC$ is the retract-closure of $\cC$ in $\psh(\cC)$.
    Since the Yoneda embedding factors over $\psh_\lex(\cC)$ and retracts of left exact presheaves are also left exact, this identifies $\idem{\cC}$ with a full subcategory of $\psh_\lex(\cC)$.

    Let $\widehat{x} \to Y \to \widehat{z}$ be a cofibre sequence in $\psh_\lex(\cC)$ with $\widehat{x}, \widehat{z} \in \idem{\cC}$.
    Then there exist $x', z' \in \idem(\cC)$ such that $\widehat{x} \oplus x' \simeq x \in \cC$ and $\widehat{z} \oplus z' \simeq z \in \cC$.
    From the cofibre sequence $x \to Y \oplus x' \oplus z' \to z$, we conclude that $Y$ is a retract of an object in $\cC$ because $\cC$ is extension-closed.
    This argument also proves the characterisation of exact sequences in $\idem{\cC}$.

    If $\widehat{u} \rightarrowtail \widehat{y} \twoheadrightarrow \widehat{w}$ is an exact sequence with $\widehat{u}, \widehat{w} \in \idem{\cU}$, then stabilising with appropriate objects from $\idem{\cU}$ as before shows that $\widehat{y} \in \idem{\cU}$ and exhibits the given sequence as a retract of an exact sequence in $\cU$.
\end{proof}

\begin{cor}
    For every exact category $\cC$ and every idempotent complete exact category $\cD$, the restriction functor
    \[ \Funex{}(\idem{\cC},\cD) \to \Funex{}(\cC,\cD) \]
    is an equivalence.
    In particular, the inclusion $\exactperf \to \exact$ of the full subcategory of idempotent complete exact categories admits a left adjoint
    \[ \idem{(-)} \colon \exact \to \exactperf. \]
\end{cor}
\begin{proof}
    This is immediate from the universal property of idempotent completion and the characterisation of exact sequences in $\idem{\cC}$ given in \cref{lem:idem}.
\end{proof}

To conclude this section, we will show that the formation of stable envelopes preserves idempotent completeness.
This can be conveniently shown using the notion of heart structures.

\begin{defn}[{Saunier \cite[Definition~2.3]{saunier:heart}}]
 A \emph{heart structure} on a stable category $\cC$ is a pair of full subcategories $(\cC_{\geq 0}, \cC_{\leq 0})$ satisfying the following properties:
 \begin{enumerate}
    \item the subcategory $\cC_{\geq 0}$ is closed under finite colimits and extensions;
    \item the subcategory $\cC_{\leq 0}$ is closed under finite limits and extensions;
    \item for every object $x \in \cC$ there exists a cofibre sequence
    \[ x_{\leq 0} \to x \to x_{\geq 1} \]
    such that $x_{\leq 0} \in \cC_{\leq 0}$ and $\Sigma^{-1}x_{\leq 1} \in \cC_{\geq 0}$.
 \end{enumerate}
 For $k \in \bbZ$, define $\cC_{\geq k} := \Sigma^n\cC_{\geq 0}$ and $\cC_{\leq k} := \Sigma^k(\cC_{\leq 0})$ and $\cC_{[k,l]} := \cC_{\geq k} \cap \cC_{\leq l}$.
 The \emph{heart} of the heart structure is defined to be
 \[ \cC^\heartsuit := \cC_{[0,0]}. \]
 The heart structure is \emph{bounded} if
 \[ \cC \simeq \bigcup_{k \in \bbN} \cC_{[-k,k]}. \]
 An exact functor $f \colon \cC \to \cD$ between stable categories with heart structures is \emph{heart-exact} if $f(\cC_{\geq 0}) \subseteq \cD_{\geq 0}$ and $f(\cC_{\leq 0}) \subseteq \cD_{\leq 0}$.
 
 Denote by $\catheart$ the category of small stable categories equipped with a bounded heart structure and heart-exact functors between them.
\end{defn}

The key observation is the following.

\begin{prop}[Sosnilo]\label{prop:sosnilo}
    Let $(\cC,\cC_{\geq 0},\cC_{\leq 0})$ be a stable category with a bounded heart structure.
    Then the following holds:
    \begin{enumerate}
        \item $(\idem{(\cC_{\geq 0})},\idem{(\cC_{\leq 0})})$ defines a bounded heart structure on $\idem{\cC}$;
        \item the induced functor $\cC^\heartsuit \to (\idem{\cC})^\heartsuit$ is an idempotent completion.
    \end{enumerate}
\end{prop}
\begin{proof}
    The full subcategory $\idem{(\cC_{\geq 0})}$ is closed under finite colimits, and $\idem{(\cC_{\leq 0})}$ is closed under finite limits.
    If $\cD \subseteq \cC$ is any extension-closed subcategory, then $\idem{\cD} \subseteq \idem{\cC}$ is also extension-closed by \cref{lem:idem}~\eqref{lem:idem-3}.

    Let now $x \in \idem{\cC}$ be arbitrary.
    We claim that there exists a cofibre sequence
    \[ x_{\leq 0} \to x \to x_{\geq 1} \]
    in $\idem{\cC}$ with $x_{\leq 0} \in \cC_{\leq 0}$ (without idempotent completion) and $x_{\geq 1} \in \idem{(\cC_{\geq 1})}$.
    
    Choose $y \in \idem{\cC}$ with $x \oplus y \in \cC$.
    Since the heart structure on $\cC$ is bounded, there exists some $k \in \bbZ$ with $x \oplus y \in \cC_{\geq k}$.
    If $k \geq 1$, then $0 \to x \oplus y \xrightarrow{\id} x \oplus y$ is a suitable cofibre sequence, so we only have to consider $k \leq 0$, where we prove the claim by induction.
    
    Choosing a decomposition of $x \oplus y$ in $\cC$ induces a bicartesian square
    \[\begin{tikzcd}
        (x \oplus y)_{\leq 0}\ar[r]\ar[d] & x\ar[d, "f"] \\
        y\ar[r, "g"] & (x \oplus y)_{\geq 1}
    \end{tikzcd}\]
    Set $x' := \cofib((x \oplus y)_{\leq 0} \to x)$.
    Since $x \oplus y \in \cC_{\geq k}$, we find that
    \[ x' \oplus \Sigma x \simeq \cofib(g+0 \colon y \oplus x \to (x \oplus y)_{\geq 1}) \in \cC_{\geq k+1}. \]
    So $x' \in \idem{(\cC_{\geq k+1})}$.
    By induction, $x'$ admits a decomposition
    \[ x'_{\leq 0} \to x' \to x'_{\geq 1} \]
    with $x'_{\leq 0} \in \cC_{\leq 0}$ and $x'_{\geq 1} \in \idem{(\cC_{\geq 1})}$.
    Consequently, the fibre of the composite map $p \colon x \to x' \to x'_{\geq 1}$ sits in a cofibre sequence $(x \oplus y)_{\leq 0} \to \fib(p) \to x'_{\leq 0}$, which shows that $\fib(p) \in \cC_{\leq 0}$.
    Consequently, $\fib(p) \to x \xrightarrow{p} x'_{\geq 1}$ is the required decomposition.

    Note that this argument can be dualised to show that every object $x \in \idem{\cC}$ also admits a decomposition
    \[ x_{\leq 0} \to x \to x_{\geq 1} \]
    in $\idem{\cC}$ with $x_{\leq 0} \in \idem{(\cC_{\leq 0})}$ and $x_{\geq 1} \in \cC_{\geq 1}$.
    
    It is immediate that $(\idem{\cC},\idem{(\cC_{\geq 0})},\idem{(\cC_{\leq 0})})$ is a bounded heart category.
    
    Note that $\cC^\heartsuit \to \cC \to \idem{\cC}$ induces a fully faithful exact functor $\idem{(\cC^\heartsuit)} \to \idem{\cC}$, which factors over $(\idem{\cC})^\heartsuit$ by definition of the heart structure on $\idem{\cC}$.
    We have to show that the resulting functor $\idem{(\cC^\heartsuit)} \to (\idem{\cC})^\heartsuit$ is essentially surjective.

    Consider first an object $x \in \cC$ with the property that $x \in (\idem{\cC})^\heartsuit$.
    Then there exist $y,z \in \cC$ such that $x \oplus y \in \cC_{\geq 0}$ and $x \oplus z \in \cC_{\leq 0}$.
    Choose decompositions
    \[ y_{\leq 0} \to y \to y_{\geq 1} \quad\text{and}\quad z_{\leq -1} \to z \to z_{\geq 0} \]
    with $y_{\leq 0} \in \cC_{\leq 0}$, $z_{\leq -1} \in \cC_{\leq -1}$, $y_{\geq 1} \in \cC_{\geq 1}$, and $z_{\geq 0} \in \cC_{\geq 0}$.
    Then we obtain cofibre sequences
    \[ \Sigma^{-1} y_{\geq 1} \to x \oplus y_{\leq 0} \to x \oplus y \quad\text{and}\quad x \oplus z \to x \oplus z_{\geq 0} \to \Sigma z_{\leq -1} \]
    showing that $x \oplus y_{\leq 0} \in \cC_{\geq 0}$ and $x \oplus z_{\geq 0} \in \cC_{\leq 0}$.
    Since $y_{\leq 0} \in \cC_{\leq 0}$, it follows that $x \oplus y_{\leq 0} \oplus z_{\geq 0} \in \cC_{\leq 0}$, and similarly $x \oplus y_{\leq 0} \oplus z_{\geq 0} \in \cC_{\geq 0}$ because $z_{\geq 0} \in \cC_{\geq 0}$.
    Hence $x$ is a summand of an object in $\cC^\heartsuit$.

    Given an arbitrary $x \in (\idem{\cC})^\heartsuit$, choose $y \in \idem{\cC}$ such that $x \oplus y \in \cC_{\leq 0}$ and pick a cofibre sequence $y_{\leq -1} \to y \to y_{\geq 0}$ with $y_{\leq -1} \in \cC_{\leq 0}$ and $y_{\geq 0} \in \idem{(\cC_{\geq 0})}$.
    Adding the cofibre sequence $0 \to x \xrightarrow{\id} x$, it follows that $x \oplus y_{\geq 0} \in \cC$.
    Moreover, $x \oplus y_{\geq 0} \in \idem{(\cC_{\geq 0})}$ because $x \in \idem{(\cC_{\geq 0})}$ by assumption.
    Now pick a cofibre sequence
    \[ (y_{\geq 0})_{\leq 0} \to y_{\geq 0} \to (y_{\geq 0})_{\geq 1} \]
    with $(y_{\geq 0})_{\leq 0} \in \idem{(\cC_{\leq 0})}$ and $(y_{\geq 0})_{\geq 1} \in \cC_{\geq 1}$.
    Then $(y_{\geq 0})_{\leq 0} \in (\idem{\cC})^\heartsuit$, and the induced cofibre sequence
    \[ x \oplus (y_{\geq 0})_{\leq 0} \to x \oplus y_{\geq 0} \to (y_{\geq 0})_{\geq 1} \]
    shows that $x \oplus (y_{\geq 0})_{\leq 0} \in \cC \cap (\idem{\cC})^\heartsuit$.
    By the preceding paragraph, this implies that $x \oplus (y_{\geq 0})_{\leq 0}$ is a retract of an object in $\cC^\heartsuit$, which implies that $x$ itself is a retract of an object in $\cC^\heartsuit$.
\end{proof}

This allows us to generalise \cite[Theorem~2.8]{bs:idempotent-completion} from exact 1-categories to arbitrary exact categories.

\begin{cor}\label{cor:derb-preserves-idempotent-complete}
    The exact functors $\derb(\idem{\cC}) \to \idem{\derb(\cC)}$ induced by the universal property of the stable envelope assemble to a natural equivalence
    \[ \derb \circ \idem{(-)} \xrightarrow{\sim} \idem{(-)} \circ \derb. \]
\end{cor}
\begin{proof}
    There is a commutative square
    \[\begin{tikzcd}
        \catperf\ar[r]\ar[d] & \exactperf\ar[d] \\
        \catst\ar[r] & \exact
    \end{tikzcd}\]
    in which $\exactperf$ denotes the full subcategory of idempotent complete exact categories, and all arrows are the obvious inclusion functors.
    The bottom horizontal functor admits a left adjoint by \cref{cor:derb-left-adjoint}, so it suffices to show that $\derb(\cC)$ is idempotent complete if $\cC$ is idempotent complete.

    For any exact category $\cC$, the induced functor $i \colon \derb(\cC)^\heartsuit \to (\idem{\derb(\cC)})^\heartsuit$ is an idempotent completion by \cref{prop:sosnilo}.
    Since $\cC$ is idempotent complete, it follows from \cite[Theorem~1.6]{sw:exact-stable} that $\derb(\cC)^\heartsuit \simeq \cC$ is idempotent complete, so $i$ is an equivalence.
    Invoking \cite[Theorem~1.6]{sw:exact-stable} a second time, it follows that $\derb(\cC) \to \idem{\derb(\cC)}$ is an equivalence as required.
\end{proof}

\section{Tensor products of exact categories}\label{sec:monoidal}
In this section, we use the Gabriel--Quillen embedding to show that the internal hom of exact categories admits a left adjoint.
The resulting tensor product of exact categories refines to a symmetric monoidal structure on the category of small exact categories.
Moreover, we show that both the Gabriel--Quillen embedding and the stable envelope construction admit symmetric monoidal refinements.

Let us first recall the internal hom of exact categories.

\begin{defn}
 Let $\cC$ and $\cD$ be small exact categories.
 Define a morphism $\tau \colon F \Rightarrow G$ in $\Fun^\exct(\cC,\cD)$ to be ingressive if the component $\tau_x \colon F(x) \to G(x)$ is ingressive in $\cD$ for all $x \in \cC$.
\end{defn}

\begin{lem}\label{lem:fun-exct-exact}
    The pair $(\Fun^\exct(\cC,\cD),\Fun^\exct(\cC,\cD)_\ing)$ is an exact category.
\end{lem}
\begin{proof}
    Let $\iota \colon F \Rightarrow G$ be ingressive in $\Fun^\exct(\cC,\cD)$.
    The cofibre $H := \cofib(\iota)$ exists in $\Fun(\cC,\cD)$, and we have to show that it is exact.
    If $i \colon x \rightarrowtail y$ is ingressive in $\cC$, then both $\iota_x$ and $\iota_{\cofib(i)}$ are ingressive, so \cite[Lemma~2.4]{sw:exact-stable} implies that the square
    \[\begin{tikzcd}
        F(x)\ar[r, tail, "\iota_x"]\ar[d, tail, "F(i)"'] & G(x)\ar[d, tail, "G(i)"] \\
        F(y)\ar[r, tail , "\iota_y"] & G(y) 
    \end{tikzcd}\]
    is Reedy cofibrant in the sense that $F(y) \cup_{F(x)} G(x) \to G(y)$ ingressive.
    Applying \cite[Lemma~2.4]{sw:exact-stable} a second time, it follows that $H(i) \colon H(x) \to H(y)$ is ingressive in $\cD$.
    Since colimits commute, the functor $H$ also preserves pushouts along ingressives, so it is exact.
    Moreover, the sequence $F \overset{\iota}{\Rightarrow} G \Rightarrow H$ is also a fibre sequence because it is so pointwise.
    
    It follows formally that every span $F_2 \overset{\tau}{\Leftarrow} F_0 \overset{\iota}{\Rightarrow} F_1$ with $\iota$ ingressive admits a pushout, that the induced morphism $F_2 \Rightarrow F_2 \cup_{F_0} F_1$ is ingressive, and that the resulting square is also a pullback.

    Suppose now that $F_0 \overset{\iota}{\Rightarrow} F_1 \overset{\pi}{\Rightarrow} F_2$ is a cofibre sequence with $\iota$ ingressive and let $\tau \colon G \Rightarrow F_2$ be arbitrary.
    Then the pullback $G \times_{F_2} F_1$ exists and fits into a cofibre sequence $F_0 \overset{\iota'}{\Rightarrow} G \times_{F_2} F_1 \Rightarrow G$ such that $\iota'$ is ingressive.
    Since the category $\Fun^\exct(\cC,\cD)$ is evidently additive, it follows that $\Fun^\exct(\cC,\cD)$ is exact.
\end{proof}

\begin{rem}\label{rem:fun-biex}
    The exact structure of \cref{lem:fun-exct-exact} has the property that currying induces an equivalence
    \[ \Fun^\biex(\cC_1 \times \cC_2,\cD) \simeq \Fun^\exct(\cC_1, \Fun^\exct(\cC_2,\cD)), \]
    where the left hand side denotes the category of functors which are exact in both variables.

    In particular, for $\cD$ Grothendieck prestable we obtain equivalences
    \begin{align*}
        \Fun^\biex(\cC_1 \times \cC_2,\cD)
        &\simeq \Fun^\Ldec(\psh_\lex(\cC_1), \Fun^\Ldec(\psh_\lex(\cC_2),\cD)) \\
        &\simeq \Fun^\Ldec(\psh_\lex(\cC_1) \otimes \psh_\lex(\cC_2), \cD).
    \end{align*}
\end{rem}

\cref{rem:fun-biex} already hints at what the construction of the tensor product should be.

\begin{const}\label{const:tenor}
 Let $\cC$ and $\cD$ be small  exact categories.
 The composite
 \[ y \colon \cC \times \cD \xrightarrow{\yo_\lex \times \yo_\lex} \psh_\lex(\cC) \times \psh_\lex(\cD) \to \psh_\lex(\cC) \otimes \psh_\lex(\cD) \]
 defines a biexact functor whose target is Grothendieck prestable by \cite[Theorem~C.4.2.1]{SAG}.

 Denote by $\cC \otimes \cD$ the smallest extension-closed subcategory of $\psh_\lex(\cC) \otimes \psh_\lex(\cD)$ containing the essential image of $y$.
 By construction, the functor $y$ induces a biexact functor
 \[ t \colon \cC \times \cD \to \cC \otimes \cD. \]
\end{const}

The following statement is the main ingredient in showing that this tensor product satisfies the desired universal property.

\begin{prop}\label{prop:tensor-univ-prop}
    The functors
    \[ I \colon \psh_\lex(\cC \otimes \cD) \to \psh_\lex(\cC) \otimes \psh_\lex(\cD) \]
    induced by the inclusion $i \colon \cC \otimes \cD \to \psh_\lex(\cC) \otimes \psh_\lex(\cD)$ and the functor
    \[ T \colon \psh_\lex(\cC) \otimes \psh_\lex(\cD) \to \psh_\lex(\cC \otimes \cD) \]    
    induced by $\yo_\lex \circ t \colon \cC \times \cD \to \psh_\lex(\cC \otimes \cD)$ are mutually inverse equivalences.

    Moreover, restriction along $t$ induces an equivalence
    \[ \Fun^\exct(\cC \otimes \cD,\cE) \xrightarrow{\sim} \Fun^\biex(\cC \times \cD, \cE) \]
    for every exact category $\cE$.
\end{prop}
\begin{proof}
    By construction, $I \circ T \simeq \id$ because $y \simeq i \circ t$.

    Since $T$ is colimit-preserving and $\cC \otimes \cD$ is closed under extensions in $\psh_\lex(\cC \otimes \cD)$, the preimage $T^{-1}(\cC \otimes \cD)$ is an extension-closed subcategory of $\psh_\lex(\cC) \otimes \psh_\lex(\cD)$.
    From the definition of $\cC \otimes \cD$, it follows that $i(\cC \otimes \cD) \subseteq T^{-1}(\cC \otimes \cD)$, ie $Ti \simeq \yo_\lex f$ for some endofunctor $f$ of $\cC \otimes \cD$.
    Hence
    \[ i \simeq ITi \simeq I \yo_\lex f \simeq i f. \]
    Since $i$ is fully faithful, it is a monomorphism in $\catinf$, so $f \simeq \id$.
    It follows that $Ti \simeq \yo_\lex$, which implies that $TI \simeq \id$.

    Now consider the commutative square
    \[\begin{tikzcd}
        \Fun^\exct(\cC \otimes \cD,\cE)\ar[r]\ar[d] & \Fun^\biex(\cC \times \cD, \cE)\ar[d] \\
        \Fun^\exct(\cC \otimes \cD,\psh_\lex(\cE))\ar[r] & \Fun^\biex(\cC \times \cD, \psh_\lex(\cE))
    \end{tikzcd}\]
    The vertical arrows are induced by $\yo_\lex$, and therefore fully faithful.
    If an exact functor $\cC \otimes \cD \to \psh_\lex(\cE)$ restricts to a biexact functor $\cC \times \cD \to \cE$, its essential image is contained in $\cE$ since $\cE$ is closed under extensions in $\psh_\lex(\cE)$.
    Hence the above square is a pullback.

    The restriction map at the bottom identifies with the equivalence
    \[ \Fun^\Ldec(\psh_\lex(\cC) \otimes \psh_\lex(\cD), \psh_\lex(\cE)) \simeq \Fun^\biex(\cC \times \cD,\psh_\lex(\cE))\]
    of \cref{rem:fun-biex}.
\end{proof}

It is now fairly straightforward to construct a symmetric monoidal structure on $\exact$.

\begin{const}
 Let $\pair$ denote the category of small categories $\cC$ equipped with a choice of wide subcategory $\cC_0$.
 This category admits finite products, and therefore carries the cartesian symmetric monoidal structure $\pair^\times$.
 
 Define $\exact^\otimes$ to be the suboperad of $\pair^\times$ determined by the following properties:
 \begin{enumerate}
     \item objects in the underlying category are small exact categories;
     \item morphisms $(\cC_1, \cC_{1,\ing}) \boxtimes \ldots \boxtimes (\cC_n, \cC_{n,\ing}) \to (\cC, \cC_{\ing})$ correspond to functors $\cC_1 \times \ldots \times \cC_n \to \cC$ which are exact in each variable.
 \end{enumerate}
\end{const}

\begin{prop}\label{cor:exact-symm-monoidal}
    Let $\cO$ be an operad.
    \begin{enumerate}
        \item\label{cor:exact-symm-monoidal-1} The operad $\exact^\otimes$ determines a symmetric monoidal structure on $\exact$ with unit object $\Proj(\bbS)$.
        \item\label{cor:exact-symm-monoidal-2} The functors
        \[ \derb_{\geq 0} \colon \exact \to \catprst \quad\text{and}\quad \derb \colon \exact \to \catst \]
        refine to symmetric monoidal functors.
        \item\label{cor:exact-symm-monoidal-3} For every $\cO$-monoidal exact category $\cC$, the inclusion functor $\cC \to \derb_{\geq 0}(\cC)$ refines to an $\cO$-monoidal exact functor such that precomposition with this inclusion induces an equivalence
        \[ \Funex{\otimes}(\derb_{\geq 0}(\cC),\cD) \xrightarrow{\sim} \Funex{\otimes}(\cC,\cD) \]
        for every small, prestably $\cO$-monoidal category $\cD$.
        The analogous assertion for $\derb(\cC)$ and small, stably $\cO$-monoidal categories $\cD$ holds as well.
        \item\label{cor:exact-symm-monoidal-4} The functors
         \[ \psh_\lex \colon \exact \to \Groth \quad\text{and}\quad \psh_\lex^\st \colon \exact \to \PrL_\st \]
        refine to symmetric monoidal functors.
        \item\label{cor:exact-symm-monoidal-5} For every $\cO$-monoidal exact category $\cC$, the Yoneda embedding refines to an $\cO$-monoidal functor $\cC \to \psh_\lex(\cC)$ such that precomposition with this embedding induces an equivalence
        \[ \Fun_{\otimes}^\Ldec(\psh_\lex(\cC),\cD) \xrightarrow{\sim} \Funex{\otimes}(\cC,\cD). \]
        for every $\cO$-monoidal Grothendieck prestable category $\cD$.
        The analogous assertion for $\psh^\st_\lex(\cC)$ presentably $\cO$-monoidal stable categories $\cD$ holds as well.
    \end{enumerate}
\end{prop}
\begin{proof}
 By construction, $\exact$ is the underlying category of the operad $\exact^\otimes$.
 Proving that $\exact^\otimes$ is a symmetric monoidal category reduces to showing that the category of biexact functors with source $\cC \times \cD$ has an initial object, which follows from \cref{prop:tensor-univ-prop}.
 Since $\Proj(\bbS)$ is the additive category freely generated by a point, it follows that
 \[\Fun^\exct(\cC \otimes \Proj(\bbS),\cD) \simeq \Fun^\oplus(\Proj(\bbS),\Fun^\exct(\cC,\cD)) \simeq \Fun^\exct(\cC,\cD), \]
 which shows that $\Proj(\bbS)$ is a unit object.
 This proves assertion~\eqref{cor:exact-symm-monoidal-1}.

 By definition, $(\catprst)^\otimes$ and $(\catst)^\otimes$ are the full suboperads of $(\catrex)^\otimes$ spanned by the prestable, respectively stable, categories.
 In particular, both identify with full suboperads of $\exact^\otimes$, making the inclusion functors $\catst \to \catprst \to \exact$ lax symmetric monoidal.
 The left adjoint $\stab \simeq \Sp^\omega \otimes - \colon \catprst \to \catst$ canonically refines to a symmetric monoidal functor, and the left adjoint $\derb_{\geq 0}$ comes equipped with a unique oplax symmetric monoidal structure by \cite[Proposition~A]{hhln:lax-adjunctions}.
 By comparing universal properties, one finds that $\derb_{\geq 0}(\Proj(\bbS)) \to \Sp^\omega_{\geq 0}$ and $\derb_{\geq 0}(\cC_1 \otimes \cC_2) \to \derb_{\geq 0}(\cC_1) \otimes \derb_{\geq 0}(\cC_2)$ are equivalences, so $\derb_{\geq 0}$ is in fact symmetric monoidal, proving assertion~\eqref{cor:exact-symm-monoidal-2}.

 It follows that the adjunction $\derb_{\geq 0} \colon \exact \rightleftarrows \catprst \cocolon \inc$ induces an adjunction
 \[ \Alg_\cO(\exact) \rightleftarrows \Alg_\cO(\catprst) \]
 for every operad $\cO$.
 Unwinding what this statement means for the unit $\cC \to \derb_{\geq 0}(\cC)$ yields assertion~\eqref{cor:exact-symm-monoidal-3}.
 For the statement about $\cC \to \derb(\cC)$, one uses that the Spanier--Whitehead stabilisation functor $\stab \simeq \sp^\omega \otimes- $ is a symmetric monoidal left adjoint to the inclusion of stable categories into prestable ones.

 By \cite[Proposition~4.8.1.3, Remark~4.8.1.8 \& Proposition~4.8.1.15]{HA}, the ind-completion functor $\Ind \colon \catrex \to \PrL$ admits a symmetric monoidal refinement.
 The lax symmetric monoidal composition
 \[ \exact \xrightarrow{\derb_{\geq 0}} \catprst \xrightarrow{\inc} \catrex \xrightarrow{\Ind} \PrL \]
 factors over the full suboperad $\Groth^\otimes$ of $(\PrL)^\otimes$.
 Since $\Ind \circ \derb_{\geq 0} \simeq \psh_\lex$ by \cref{cor:ind-derb-is-plex}, this provides a lax symmetric monoidal refinement of $\psh_\lex$.
 The universal property of $\psh_\lex$ from \cref{thm:plex}~\eqref{thm:plex-5} implies that $\psh_\lex \colon \exact \to \Groth$ is symmetric monoidal.
 The stabilisation functor $\Sp(-) \simeq \Sp \otimes - \colon \Groth \to \PrL_\st$ is also symmetric monoidal, so $\psh_\lex^\st$ carries an induced symmetric monoidal structure.
 This proves assertion~\eqref{cor:exact-symm-monoidal-4}.

 For the last assertion, note that the restriction functors
 \[ \Fun^\Ldec_\otimes(\psh_\lex(\cC),\cD) \simeq \Fun^\Ldec_\otimes(\Ind(\derb_{\geq 0}(\cC)),\cD) \to \Funex{\otimes}(\derb_{\geq 0}(\cC),\cD) \to \Funex{\otimes}(\cC,\cD) \]
 are equivalences due to \cite[Proposition~4.8.1.10]{HA} and the version of assertion~\eqref{cor:exact-symm-monoidal-4} for large categories.
 By virtue of \cref{prop:derb-pos-left-adjoint}, the large version of $\derb_{\geq 0}$ coincides with its small incarnation on small categories, proving the desired universal property.

 The analogous claim for $\psh^\st_\lex$ follows since the stabilisation functor $\Sp \otimes - \colon \PrL \to \PrL_\st$ is symmetric monoidal.
\end{proof}

For a fixed exact category $\cD$, the functor $- \otimes \cD$ need not preserve fully faithful functors: for example, $- \otimes \Sp^\omega \simeq \derb(-)$.
However, it does preserve inclusions of left special subcategories.

\begin{cor}\label{cor:tensor-special}
    Let $\cU$ be a left special subcategory of a small exact category $\cC$.
    Then for every small exact category $\cD$, the induced functor
    \[ \cU \otimes \cD \to \cC \otimes \cD \]
    exhibits the domain as a left special subcategory.
\end{cor}
\begin{proof}
 \cref{prop:special-fullyfaithful} asserts that $j_\natural \colon \psh_\lex(\cU) \to \psh_\lex(\cC)$ is a left adjoint in $\PrL$.
 Consequently, the induced functor
 \[ \psh_\lex(\cU) \otimes \psh_\lex(\cD) \xrightarrow{j_\natural \otimes \id} \psh_\lex(\cC) \otimes \psh_\lex(\cD)\]
 on Lurie tensor products is fully faithful and left adjoint in $\PrL$.
 Since $\psh_\lex$ is symmetric monoidal by \cref{cor:exact-symm-monoidal}, this implies that $\cU \otimes \cD$ is left special in $\cC \otimes \cD$, again by \cref{prop:special-fullyfaithful}.
\end{proof}

\section{The example of synthetic spectra}\label{sec:synthetic}
In this section, we explain in which sense Pstr\k{a}gowski's categories of synthetic spectra are presentable stable envelopes of exact categories.
This was first pointed out to the authors by Robert Burklund.

\begin{defn}[{Pstr\k{a}gowski \cite[Definition~2.3]{pstragowski:synthetic}}]
    An \emph{additive site} is a small site $\mathcal{C}$ such that the underlying category is additive and such that every covering consists of a single morphism.
\end{defn}

\begin{defn}[{Pstr\k{a}gowski \cite[Definition~3.13]{pstragowski:synthetic}}]
    Let $E$ be a homotopy associative ring spectrum.
    We say that a spectrum $X$ is \emph{$E$-finite projective} if it is finite and $E_{*}X$ is a finite projective $E_{*}$-module.
    We denote the full subcategory of spectra spanned by $E$-finite projectives by $\sp^{\mathrm{fp}}_{E}$.
\end{defn}

\begin{ex}[{Pstr\k{a}gowski \cite[Proposition~3.23]{pstragowski:synthetic}}]\label{syn:topology}
    If $E$ is a homotopy associative ring spectrum, one can define a Grothendieck topology on $\sp^{\mathrm{fp}}_{E}$ where coverings consist of a single morphism $x\to y$ such that $E_{*}x\to E_{*}y$ is surjective. The category $\sp^{\mathrm{fp}}_{E}$ of $E$-finite projectives forms an additive site. 
\end{ex}

\begin{prop}\label{syn:exact-gives-top}\label{syn:cov-cowaldhausen-structure}
    There is an isomorphism of posets
    \[
    \{\text{coWaldhausen structures on }\cC\} \xrightarrow{\sim} \left\{\begin{matrix}
        \text{Additive quasi-topologies on }\cC\\ \text{such that }x\to 0\text{ is a covering}
    \end{matrix}\right\}.
    \]
    which associates to a coWaldhausen structure the quasi-topology for which a covering is precisely an egressive morphism.
\end{prop}
\begin{proof}
    Every coWaldhausen structure defines a quasi-topology in the specified way because egressive morphisms form a subcategory and are closed under pullbacks along arbitrary maps.
    Declaring the coverings of an additive quasi-topology to be egressive defines a coWaldhausen structure if $x \to 0$ is a covering for all $x \in \cC$.
    By inspection, these assignments are mutually inverse.
\end{proof}

\begin{rem}
    It is currently unknown to the authors whether there is some (non-tautological) condition one can put on an additive site such that the coWaldhausen structure of \Cref{syn:cov-cowaldhausen-structure} is exact. 
\end{rem}

\begin{theorem}[{Pstr\k{a}gowski \cite[Theorem~2.8]{pstragowski:synthetic}}]\label{syn:recognition}
    Let $\mathcal{C}$ be an additive site.
    An additive presheaf $X \in \psh_{\Sigma}(\mathcal{C})$ is a sheaf if and only if for every fibre sequence $F \to B \to A$ where $B \to A$ is a covering, the induced sequence $X(A) \to X(B) \to X(F)$ is a fibre sequence of anima.
\end{theorem}

\begin{cor}
    Let $\mathcal{C}$ be an exact category. A functor $X \in \psh_{\Sigma}(\mathcal{C})$ is a left exact presheaf if and only if it is a sheaf for the egressive topology defined in \cref{syn:exact-gives-top}. In particular, we get that $\psh_{\lex}(\mathcal{C})=\shv_{\Sigma}(\mathcal{C})$.
\end{cor}

\begin{rem}
    The fact that $\psh_\lex(\cC) = \shv_{\Sigma}(\cC)$ was already shown as part of proving \Cref{thm:plex}. We reiterate it here to make the connection to the theory of additive sheaves on additive sites more visible.
\end{rem}

\begin{defn}[{Pstr\k{a}gowski \cite[Definition~3.14]{pstragowski:synthetic}}]
    Let $E$ be a homotopy associative ring spectrum.
    We say that $E$ is \emph{Adams-type} if
    \begin{enumerate}
        \item $E$ can be written as a filtered colimit $\colim_{\alpha} E_{\alpha} \simeq E$ such that
        \item every $E_{\alpha}$ is $E$-finite projective and the canonical map
        \[
            E^{*}E_{\alpha}\to \mathrm{Hom}_{E_*}(E_{*}E_{\alpha},E_{*})
        \]
        is an isomorphism.
    \end{enumerate}
\end{defn}

\begin{defn}[{\cite[Definition~4.1]{pstragowski:synthetic}}]
    An \emph{$E$-based synthetic spectrum} is an additive sheaf of spectra on the site $\sp^{\mathrm{fp}}_{E}$ with the topology described in \cref{syn:topology}. We denote the category of $E$-based synthetic spectra by $\Syn_{E}$. If $y\colon \sp^{\mathrm{fp}}_{E}\to \shv_{\Sigma}(\sp^{\mathrm{fp}}_{E})$ denotes the Yoneda embedding into the category of additive sheaves of spaces, then we let $\nu\colon \sp^{\mathrm{fp}}_{E}\to \Syn_{E}$ be the composite of $y$ with
        \[
            \Sigma^{\infty}\colon \shv_\Sigma(\sp^{\mathrm{fp}}_{E})\to \shv_{\Sigma}^{\sp}(\sp^{\mathrm{fp}}_E) =\hspace{-1.5pt}\cocolon{}  \Syn_{E}.
        \]
    We name the composite the \emph{synthetic analogue}.
\end{defn}

\begin{rem}
    The synthetic analogue extends to a functor $\nu\colon \sp \to \Syn_{E}$ on all spectra. Usually, it is this functor which is called the synthetic analogue.
    However, for this note only the above restricted version is necessary.
\end{rem}

\begin{theorem}[{Pstr\k{a}gowski \cite[Proposition~4.2]{pstragowski:synthetic}}]\label{thm:synthetic-presentably-monoidal}
    The category of $E$-based synthetic spectra is a presentably symmetric monoidal $\infty$-category generated by the image of the synthetic analogue.
\end{theorem}

We will now explain how \cref{thm:synthetic-presentably-monoidal} also follows from the results of \cref{sec:gabriel-quillen}.

\begin{theorem}\label{syn:E-based-synthetic-exact}
    Let $E$ be a homotopy associative ring spectrum of Adams-type. 
    Then the coWaldhausen structure determined by the additive site $\sp^{\mathrm{fp}}_{E}$ defines an exact category. Furthermore, the synthetic analogue induces a symmetric monoidal equivalence
    \[ \Syn_E \simeq \psh^{\st}_\lex(\sp^{\mathrm{fp}}_E).\]
\end{theorem}
\begin{proof}
    To see that the coWaldhausen structure defines an exact structure on $\sp^{\mathrm{fp}}_{E}$, it suffices to prove the following:
    \begin{enumerate}
         \item\label{syn:E-based-synthetic-exact-1} The category $\sp^{\mathrm{fp}}_{E}$ is additive.
         \item\label{syn:E-based-synthetic-exact-2} If $f \to x \twoheadrightarrow y$ is a fibre sequence and $x \twoheadrightarrow y$ is egressive, and $f \to z$ is any map in $\sp^{\mathrm{fp}}_{E}$, then the pushout $x \oplus_{f} z$ exists, and there exists an egressive map $x \oplus_{f} z \twoheadrightarrow w$ such that 
         $$
            z \to x \oplus_{f} z \twoheadrightarrow w
         $$
         is a fibre sequence.
         \item\label{syn:E-based-synthetic-exact-3} If $f \to x \twoheadrightarrow y$ is a fibre sequence and $x \twoheadrightarrow y$ is egressive, then it is also a cofibre sequence.
     \end{enumerate} 
     Property~\eqref{syn:E-based-synthetic-exact-1} is clear.

     Ad~\eqref{syn:E-based-synthetic-exact-2}: Suppose $f \to x \twoheadrightarrow y$ is a fibre sequence and $f \to z$ is any map. We first argue that the pushout $x \oplus_{f}z$ exists in $\sp^{\mathrm{fp}}_{E}$. We know from \cite[Lemma 3.21]{pstragowski:synthetic} that every object in $\sp^{\mathrm{fp}}_{E}$ is dualisable, with dual given by the Spanier--Whitehead dual.
     Spanier--Whitehead duality takes maps which are injective on $E$-homology to maps which are surjective. In particular, since $Dx \to Df$ is surjective, it follows that $Dx \times_{Df} Dz$ is $E$-finite projective and thus so is $D(Dx \times_{Df}Dz) \simeq x \oplus_{f}z$. Now, consider the diagram
     \[
        \begin{tikzcd}
            f\ar[r]\ar[d] & x\ar[d]\ar[r, twoheadrightarrow] & y\ar[d, "\sim"] \\
            z\ar[r] & x \oplus_{f}z \ar[r] & \operatorname{cof}(z \to x \oplus_{f} z).
        \end{tikzcd}
     \]
     Since the composite $x \to x \oplus_{f} z \to \operatorname{cof}(z \to x \oplus_{f} z)$ is surjective on $E$-homology, it follows that so is $x \oplus_{f} z \to \operatorname{cof}(z \to x \oplus_{f} z)$. It is clear that 
     \[
        z \to x \oplus_{f} z \twoheadrightarrow \operatorname{cof}(z \to x \oplus_{f} z)
     \]
     is a fibre sequence.

     Ad~\eqref{syn:E-based-synthetic-exact-3}: Let $f \to x \twoheadrightarrow y$ be a fibre sequence. Note that it is a fibre sequence in $\sp$, and hence a cofibre sequence $\sp$. It follows that it is a cofibre sequence in $\sp^{\mathrm{fp}}_{E}$.

     The synthetic analogue $\nu \colon \Sp^{\mathrm{fp}}_E \to \Syn_E$ refines to a symmetric monoidal functor by \cite[Corollary~2.29 \& Lemma~3.21]{pstragowski:synthetic}.
     Using \cite[Proposition~4.2]{pstragowski:synthetic} and \cref{cor:exact-symm-monoidal}~\eqref{cor:exact-symm-monoidal-5} provides a symmetric monoidal, colimit-preserving functor
     \[ \psh^\st_\lex(\Sp^\mathrm{fp}_E) \to \Syn_E. \]
     Due to \cref{syn:recognition}, the underlying functor is an equivalence.
\end{proof}

\begin{rem}
    As a result of \Cref{syn:E-based-synthetic-exact}, \Cref{thm:synthetic-presentably-monoidal} may be regarded as a special case of \cref{thm:gabriel-quillen} and \cref{cor:exact-symm-monoidal}.
\end{rem}

\begin{rem}
    It also follows from \cite[Lemma~4.23]{pstragowski:synthetic} that $\Sp^{\mathrm{fp}}_E$ forms an extension-closed subcategory of $\Syn_E$. This provides an alternate proof of \cref{syn:E-based-synthetic-exact}, but would make the entire discussion circular.
\end{rem}

\section{Localising invariants}\label{sec:delooping}

As another application of our results, we will now introduce a fairly natural notion of localising invariants for exact categories and show that it passes the obvious sanity checks: every localising invariant valued in connective spectra deloops uniquely to spectrum-valued localising invariant, and algebraic K-theory is the universal localising invariant under the groupoid core.

\begin{defn}
    Call a reduced functor $F \colon \exact \to \cX$
    \begin{enumerate}
        \item \emph{left localising} if for every left special subcategory inclusion $\cU \subseteq \cC$ the induced square
        \[\begin{tikzcd}
            F(\cU)\ar[r]\ar[d] & F(\cC)\ar[d] \\
            * \simeq F(0)\ar[r] & F(\cC/\cU)
        \end{tikzcd}\]
        is a pullback.
        \item \emph{right localising} if $F \circ (-)^\op \colon \exact \to \cX$ is left localising.
    \end{enumerate}
    Denote by $\Fun^\lloc(\exact,\cX)$ and $\Fun^\rloc(\exact,\cX)$ the full subcategories of the functor category $\Fun(\exact,\cX)$ spanned by the left localising and right localising functors, respectively.
\end{defn}

\begin{rem}\label{rem:localising-idem}
    Since $\cC \to \idem{\cC}$ is the inclusion of a left special subcategory, it follows that left (or right) localising functors are invariant under idempotent completion.
\end{rem}

\begin{lem}\label{lem:derb-karoubi}
    Let $\cU$ be a left special subcategory of an exact category $\cC$.
    Then
    \[ \idem{\derb(\cU)} \to \idem{\derb(\cC)} \to \idem{\derb(\cC/\cU)} \]
    is a Karoubi sequence of stable categories.
\end{lem}
\begin{proof}
    The induced functor $\derb(\cU) \to \derb(\cC)$ is also fully faithful by \cref{rem:derb-special}.
     Since $\derb \colon \exact \to \catst$ is left adjoint by \cref{cor:derb-left-adjoint},
    \[ \derb(\cU) \to \derb(\cC) \to \derb(\cC/\cU) \]
    is a cofibre sequence of stable categories.
    Taking idempotent completions yields a Karoubi sequence.
\end{proof}

\begin{cor}\label{cor:k-localising}
    The functor
    \[ \K \circ \idem{(-)} \colon \exact \to \Sp_{\geq 0} \]
    is both left and right localising.
\end{cor}
\begin{proof}
    Since K-theory is invariant under taking opposites (see eg \cite[Corollary~5.16.1]{barwick:heart}), it suffices to prove that $\K \circ \idem{(-)}$ is left localising.
     
    The localisation theorem for stable categories together with \cref{lem:derb-karoubi} implies that
    \[ \K(\idem{\derb(\cU)}) \to \K(\idem{\derb(\cC)}) \to \K(\idem{\derb(\cC/\cU)}) \]
    is a fibre sequence of connective spectra.
    Since $\derb$ commutes with idempotent completion by \cref{cor:derb-preserves-idempotent-complete}, the Gillet--Waldhausen theorem in connective K-theory \cite[Theorem~1.7]{sw:exact-stable} implies that $\K \circ \idem{(-)}$ is left localising.
\end{proof}

We are now able to prove the following variant of \cite[Theorem~4.4.3]{ck:a1-invariance}.
Note that our argument differs from \textit{loc.\ cit.} because instead of the Bass delooping we deloop by tensoring with the ``Calkin category'' of the sphere.
While we consider this construction to be well-known to the experts (at least in the case of stable categories), we are not aware of a citable source in the literature for this type of argument.

\begin{theorem}\label{thm:delooping}
    \ \begin{enumerate}
        \item
        If $\cX$ is a pointed and finitely complete category,
        then $\Fun^\lloc(\exact,\cX)$ and $\Fun^\rloc(\exact,\cX)$ are stable.
        If $\cX$ admits filtered colimits and these commute with finite limits, the same is true for the full subcategories of finitary functors.
        \item If $\cY$ is a stable category with a right complete t-structure, then taking connective covers induces equivalences
        \[ \tau_{\geq 0} \colon \Fun^\lloc(\exact,\cY) \to \Fun^\lloc(\exact, \cY_{\geq 0}) \]
        and
        \[ \tau_{\geq 0} \colon \Fun^\rloc(\exact,\cY) \to \Fun^\rloc(\exact, \cY_{\geq 0}). \]
        If $\cY_{\leq 0}$ is closed under filtered colimits, the same is true for the subcategories of finitary functors.
    \end{enumerate}
\end{theorem}
\begin{proof}
    Precomposition with $(-)^\op$ induces an equivalence between the categories of left localising and right localising functors, so we only have to consider left localising functors.
    
    Since $\Fun^\lloc(\exact, \cX_{\geq 0})$ is pointed and has finite limits (which are computed pointwise), stability follows once we show that $\Omega$ is invertible.

    Let $\cF$ be the additive category of countably generated projective $\bbS$-modules.
    This contains the category $\Proj(\bbS)$ of finitely generated free $\bbS$-modules as a full additive subcategory.
    Set $\cQ := \cF / \Proj(\bbS)$.
    By \cref{cor:tensor-special}, we obtain for every left localising functor a natural fibre sequence
    \[ F(-) \to F(- \otimes \cF) \to F(- \otimes \cQ) \]
    because the tensor product commutes with colimits in both variables.
    Since $\cF$ admits a functorial Eilenberg swindle, the same is true for $\cC \otimes \cF$ for any exact category $\cC$.
    Hence $\Omega F(- \otimes \cQ) \simeq F$.
    Since precomposition with $- \otimes \cQ$ and postcomposition with $\Omega$ commute, this shows that $\Omega$ has both a left and a right inverse.

    If filtered colimits commute with finite limits, then the loops functor on finitary, left localising functors is still given by pointwise application of $\Omega$.
    Since $- \otimes \cQ$ commutes with colimits, it follows that the finitary, left localising functors form a stable subcategory of all left localising functors.

    The statement about left localising functors $\exact \to \cY$ is immediate from the stability of $\Fun^\lloc(\exact, \cY_{\geq 0})$ since
    \begin{align*}
        \Fun^\lloc(\exact,\cY)
        &\simeq \Fun^\lloc(\exact, \lim_\Omega \cY_{\geq 0}) \\
        &\simeq \lim_\Omega \Fun^\lloc(\exact,\cY_{\geq 0}),
    \end{align*}
    where $\lim_\Omega$ denotes the sequential limit over iterates of the loops functor.
    Moreover, observe that any left localising functor $F \colon \exact \to \cY$ satisfies
    \begin{align*}
        F \simeq \colim_n \tau_{\geq -n}F \simeq \colim_n \Omega^n\tau_{\geq 0}\Sigma^n F \simeq \colim_n \Omega^n\tau_{\geq 0}F(- \otimes \cQ^{\otimes n})
    \end{align*}
    by right completeness of the t-structure and the preceding discussion.
    If $\tau_{\geq 0}$ preserves filtered colimits, it follows that $F$ is finitary if and only if $\tau_{\geq 0}F$ is finitary.
\end{proof}

\begin{rem}
 Let $\cX$ be a stable category with a right complete t-structure and
 let $F \colon \exact \to \cX_{\geq 0}$ be a left localising functor.
 Suppose that there exists a natural equivalence $F \simeq F \circ (-)^\op$.
 Then $F$ is both left and right localising, so there exist two a priori different deloopings $F^\lloc$ and $F^\rloc$ of $F$ as a left localising and right localising functor, respectively.

 However, $F^\rloc \circ (-)^\op$ is by definition a left localising invariant satisfying
 \[ \tau_{\geq 0} \circ F^\rloc \circ (-)^\op \simeq F \circ (-)^\op \simeq F, \]
 so $F^\rloc \simeq F^\lloc$.
 In particular, such a functor admits a unique delooping which is both left and right localising and is invariant under taking opposite categories.
 Note, however, that this identification does depend on the choice of equivalence $F \simeq F \circ (-)^\op$.
\end{rem}

\begin{defn}
    Define the non-connective algebraic K-theory functor
    \[ \bbK \colon \exact \to \Sp \]
    as the (left or right) localising functor satisfying $\tau_{\geq 0}\bbK \simeq \K \circ \idem{(-)}$.
\end{defn}

\begin{prop}\label{prop:agreement}
    \ \begin{enumerate}
        \item The functor $\bbK$ coincides with Schlichting's non-connective algebraic K-theory functor on exact 1-categories introduced in \cite{schlichting:delooping}.
        \item The functor $\bbK$ coincides with the non-connective algebraic K-theory functor on stable categories defined by Blumberg, Gepner and Tabuada in \cite{bgt:alg-k}.
    \end{enumerate}
\end{prop}
\begin{proof}
    The proof that Schlichting's delooping coincides with the canonical one amounts to an alternative argument for \cref{thm:delooping} using a generalisation of Schlichting's construction.

    Let $\cC$ be an exact category.
    Define $\cF\cC \subseteq \psh_\lex(\cC)$ as the full subcategory on those objects which are equivalent to a sequential colimit
    \[ \yo_\lex(x_0) \to \yo_\lex(x_1) \to \yo_\lex(x_2) \to \ldots \]
    along (images of) ingressives in $\cC$.
    In the following, we will omit the Yoneda embedding $\yo_\lex$ from notation.
    
    To show that $\cF\cC$ is extension-closed, consider a cofibre sequence $X \to Y \xrightarrow{p} Z$ and choose diagrams $x_0 \rightarrowtail x_1 \rightarrowtail x_2 \rightarrowtail \ldots$ and $z_0 \rightarrow z_1 \rightarrowtail z_2 \rightarrowtail \ldots$ whose colimits are $X$ and $Z$, respectively.
    Pulling back the filtration of $Z$ along $p$, we obtain a diagram
    \[ Y \times_Z z_0 \to Y \times_Z z_1 \to Y \times_Z z_2 \to \ldots \]
    whose colimit is $Y$.
    For fixed $n$, we have a cofibre sequence
    \[ X \to Y \times_Z z_n \to z_n. \]
    Since $\cC$ is left special in $\psh_\lex(\cC)$, there exists a pushout
    \[\begin{tikzcd}
     y_n'\ar[r, tail]\ar[d] & y_n\ar[d] \\
     X\ar[r] & Y \times_Z z_n
     \end{tikzcd}\]
     with $y_n' \rightarrowtail y_n$ ingressive in $\cC$.
    Since $X \simeq \colim_n x_n$, the morphism $y_n'\to X$ factors over some $x_{k_n}$, and we conclude that
    \[ Y \times_Z z_n \simeq \colim_{k \geq k_n} \left( x_k \cup_{y_n'} y_n \right). \]
    Without loss of generality, we may assume that the sequence $(k_n)_n$ is strictly increasing.
    This presents $Y$ as the colimit of a diagram
    \[ \colim_{k \geq k_0} \left( x_k \cup_{y_0'} y_0 \right) \to \colim_{k \geq k_1} \left( x_k \cup_{y_1'} y_1 \right) \to \colim_{k \geq k_2} \left( x_k \cup_{y_2'} y_2 \right) \to \ldots \]
    We define another sequence $(l_n)_n$ by induction.
    Set $l_0 := k_0$. Assuming $l_n$ has been defined, choose $l_{n+1} >l_n$ such that
    \begin{align*}
        l_{n+1} > \min \{ k \mid k \geq k_{n+1} \text{ and }& x_{l_n} \cup_{y_n'} y_n \to \colim_{k \geq k_{n+1}} (x_k \cup_{y_{n+1}'} y_{n+1}) \\ &\quad\text{factors over } x_{l_{n+1}} \cup_{y_{n+1}'} y_{n+1} \}.
    \end{align*}
    Then the objects $a_n := x_{l_n} \cup_{y_n'} y_n$ assemble into a directed system
    \[ a_0 \to a_1 \to a_2 \to \ldots \]
    Each connecting map in this system fits into a diagram of exact sequences
    \[\begin{tikzcd}
        x_{l_n}\ar[r, tail]\ar[d, tail] & a_n\ar[r, two heads]\ar[d] & z_n\ar[d, tail] \\
        x_{l_{n+1}}\ar[r, tail] & a_{n+1}\ar[r, two heads] & z_{n+1}
    \end{tikzcd}\]
    Since the two outer vertical arrows are ingressive as indicated, the middle vertical is also ingressive.
    Taking the colimit of these diagrams over $n$, we obtain a cofibre sequence
    \[ \colim_n x_{l_n}\to \colim_n a_n\to \colim_n z_n \]
    which maps to the exact sequence $X \to Y \to Z$.
    Since this transformation is an equivalence on the two outer terms, it follows that $Y \simeq \colim_n a_n$.
    Hence $\cF\cC$ is an extension-closed subcategory of $\psh_\lex(\cC)$.

    Since $\cC$ is left special in $\psh_\lex(\cC)$, it is also left special in $\cF\cC$. Denote by $\cQ\cC := \cF\cC/\cC$ the quotient exact category.
    Then we obtain a functorial fibre sequence
    \[ \K(\idem{\cC}) \to \K(\idem{(\cF\cC)}) \to \K(\idem{(\cQ\cC)}) \]
    by \cref{cor:k-localising}.
    As in the proof of \cref{thm:delooping}, it follows that
    \[ \bbK(\cC) \simeq \colim_n \Omega^n\K(\idem{(\cQ^n\cC)}) \simeq \colim_n \Omega^n\K(\cQ^n\cC)), \]
    where the second equivalence is a consequence of the cofinality theorem \cite[Theorem~10.11]{barwick:algK}.
    
    In the case of an exact 1-category, $\cF\cC$ is exactly the ``countable envelope'' used in \cite[Section~3]{schlichting:delooping}, see also \cite[Appendix~B]{keller:stable-cats}.
    It follows from \cite[Lemma~3.2]{schlichting:delooping} and \cite[Theorem~3.12]{winges:localisation} that the quotient exact category $\cQ$ is an exact 1-category, so this construction reproduces precisely Schlichting's definition of non-connective algebraic K-theory in \cite{schlichting:delooping}.

    For $\cC$ a stable category, $\cF\cC$ coincides with the full subcategory of $\aleph_1$-compact objects in $\Ind(\cC)$ because every countable filtered diagram receives a cofinal functor from the poset $\bbN$.
    Hence, this construction simultaneously reproduces the construction of non-connective algebraic K-theory from \cite[Section~9]{bgt:alg-k}.
\end{proof}

\begin{cor}\label{cor:nc-K-universal-property}
    Non-connective algebraic K-theory is the initial left localising, $\Sp$-valued invariant under $\Sigma^\infty \iota$.
\end{cor}
\begin{proof}
 Observing that every left localising invariant is in particular additive,
 this follows from \cref{thm:delooping} and the universality of $\K$ (see eg \cite[Theorem~5.1]{hls:localisation}): for every left localising invariant $F \colon \exact \to \Sp$, we have
 \[ \Nat(\bbK,F) \simeq \Nat(\K,\tau_{\geq 0}F) \simeq \Nat(\Sigma^\infty \iota, \tau_{\geq 0}F) \simeq \Nat(\Sigma^\infty \iota,F). \qedhere \]
\end{proof}

\begin{cor}[Gillet--Waldhausen]\label{cor:gillet-waldhausen}
    Let $\cX$ be a stable category with a right complete t-structure and let $F \colon \exact \to \cX_{\geq 0}$  be a left localising invariant.
    \begin{enumerate}
        \item If the unit transformation $\cC \to \derb(\cC)$ induces an equivalence $F(\cC) \to F(\derb(\cC))$ for all exact categories $\cC$, then the same is true for the non-connective delooping of $F$.
        \item Assume in addition that $\cX_{\leq 0}$ is closed under filtered colimits and that $F$ is finitary.
        Then $F(\cC) \to F(\derb(\cC))$ is an equivalence for every small exact category $\cC$.
        \item The unit transformation induces an equivalence
        \[ \bbK(\cC) \xrightarrow{\sim} \bbK(\derb(\cC)) \]
        for every exact category $\cC$.
    \end{enumerate}
\end{cor}
\begin{proof}
    \cref{lem:derb-karoubi} and \cref{rem:localising-idem} imply that $F \circ \derb$ is left localising, so the first assertion follows by observing that taking connective covers of left localising $\cX$-valued functors is conservative by \cref{thm:delooping}.

    For the second assertion, the inclusion $\cC \to \derb_{\geq 0}(\cC)$ is left special by \cref{lem:pshlex-special}.
    Since $\derb_{\geq 0}\cC$ is generated by $\cC$ under finite colimits, the quotient of $\derb_{\geq 0}(\cC)$ by $\cC$ is trivial.
    It follows that $F(\cC) \to F(\derb_{\geq 0}(\cC))$ is an equivalence because $F$ is left localising.
    Moreover, $F$ is invariant under Spanier--Whitehead stabilisation because it preserves filtered colimits.
    
    The assertion about algebraic K-theory follows by combining the first and second assertion with \cref{cor:k-localising}.
\end{proof}

\begin{cor}\label{cor:k-lax-monoidal}
    The non-connective algebraic K-theory functor
    \[ \bbK \colon \exact \to \Sp \]
    admits a lax symmetric monoidal refinement.
\end{cor}
\begin{proof}
    Non-connective algebraic K-theory is known to refine to a lax symmetric monoidal functor $\catst \to \Sp$ by \cite[Corollary~1.6]{bgt-multiplicative}; note that \textit{loc.\ cit.} speaks about the same functor as we do by \cref{prop:agreement}.
    Since $\derb$ is symmetric monoidal by \cref{cor:exact-symm-monoidal}~\eqref{cor:exact-symm-monoidal-2}, the corollary follows from \cref{cor:gillet-waldhausen}.
\end{proof}

\bibliographystyle{alpha}
\bibliography{gabrielquillen}

\end{document}